\documentclass[11pt,reqno]{amsart}
\usepackage{graphicx,amssymb,mathrsfs,amsmath,amsfonts,appendix}
\usepackage{subfigure}
\usepackage{amsthm}
\usepackage{paralist}
\usepackage{enumitem}
\usepackage[utf8]{inputenc}
\usepackage{comment}
\usepackage{amsmath}
\usepackage{mathtools}
\usepackage{esint}
\usepackage[colorlinks=true, pdfstartview=FitV, linkcolor=blue, citecolor=blue, urlcolor=blue]{hyperref}
\numberwithin{equation}{section}

\newtheorem{lemma}{Lemma}
\newtheorem{theorem}{Theorem}[section]
\newtheorem{definition}{Definition}
\newtheorem{remark}{Remark}[section]

\usepackage[colorinlistoftodos]{todonotes}

\newcommand{\xd}{\textrm{d}}
\newcommand{\bB}[0]{{\bf B}}

\newcommand{\divergence}[0]{\operatorname{div}}
\newcommand{\ep}[0]{\varepsilon}

\title[Degenerate nonlocal CH equations]{Degenerate nonlocal Cahn-Hilliard equations:
well-posedness, regularity and 
local asymptotics}
\author[Elisa Davoli]{Elisa Davoli}
\address{Institut f\"ur Mathematik, University of Vienna, Oskar-Morgenstern-Platz 1, 1090 Vienna, Austria}
\email{elisa.davoli@univie.ac.at}
\author[Helene Ranetbauer]{Helene Ranetbauer}
\address{Institut f\"ur Mathematik, University of Vienna, Oskar-Morgenstern-Platz 1, 1090 Vienna, Austria}
\email{helene.ranetbauer@univie.ac.at}
\author[Luca Scarpa]{Luca Scarpa}
\address{Institut f\"ur Mathematik, University of Vienna, Oskar-Morgenstern-Platz 1, 1090 Vienna, Austria}
\email{luca.scarpa@univie.ac.at}
\author[Lara Trussardi]{Lara Trussardi}
\address{Institut f\"ur Mathematik, University of Vienna, Oskar-Morgenstern-Platz 1, 1090 Vienna, Austria}
\email{lara.trussardi@univie.ac.at}

\keywords{Nonlocal Cahn-Hilliard equation, degenerate potential, singular kernel, regularity, well-posedness, nonlocal-to-local convergence, convection}
\subjclass[2010]{45K05, 35K25, 35K55, 35B40, 76R05}
\begin{document}

\begin{abstract}
Existence and uniqueness of solutions for nonlocal Cahn-Hilliard equations with degenerate potential is shown. The nonlocality is described by means of a symmetric singular kernel not falling within the framework of any previous existence theory. A convection term is also taken into account. Building upon this novel existence result, we prove convergence of solutions for this class of nonlocal Cahn-Hilliard equations to their local counterparts, as the nonlocal convolution kernels approximate a Dirac delta. Eventually, we show that, under suitable assumptions on the data, the solutions to the nonlocal Cahn-Hilliard equations exhibit further regularity, and the nonlocal-to-local convergence is verified in a stronger topology. 
\end{abstract}

\maketitle

\tableofcontents

\section{Introduction}
\label{sec:intro}

The Cahn-Hilliard equation was originally introduced in \cite{CH} in order to model the so-called ``spinodal decomposition'' phenomenon
occurring during the phase separation processes in binary metallic alloys. Since then it has acquired fundamental importance in several diffuse-interface models
in different fields, ranging from physics and engineering to biology.

This nonlinear parabolic PDE exhibits a gradient-flow structure (in the $H^{-1}$-metric) in terms of the free energy functional given by, cf.~\cite{CH},
\begin{equation}
\label{eq:enfunCH}
E_{CH}(u)=\int_{\Omega}\Bigl(\frac{\tau^{2}}{2}|\nabla u(x)|^{2}+F(u(x))\Bigr) \,\xd x,%
\end{equation}
where $\Omega$ is the $d$-dimensional flat torus, $F$ is a double-well potential, and $\tau$ is a small positive parameter related to the thickness of the transition region. The choice of the set $\Omega$ is classical in the literature, and corresponds to imposing periodic boundary conditions.
The corresponding evolution problem reads as follows
\begin{align}\label{eq:CH}
\begin{aligned}
\partial_t u+{\rm div}\, J_{CH} & =0\text{,}\\ 
J_{CH} & =-m(u)\nabla \mu_{CH}\text{,}\\
\mu_{CH} & =\frac{\delta E_{CH}(u)}{\delta u}=-\tau^{2}\Delta u+F^{\prime}(u)\text{,}
\end{aligned}
\end{align}
where $\mu_{CH}$ is the chemical potential associated to the energy $E_{CH}$, and the symbol ${\rm div}(\cdot)$ denotes the divergence operator. The function $m(\cdot)$ in~\eqref{eq:CH} is known as mobility. \\

The mathematical literature on the classical Cahn-Hilliard
equation has been widely developed in the last decades,
in terms of well-posedness of the system with possibly 
degenerate potentials, viscosity terms and dynamic boundary conditions, but also in the direction of regularity, long-time behaviour of solutions, and optimal control problems. Among the extensive literature, we mention
the works \cite{cher-gat-mir, cher-mir-zel, cher-pet, colli-fuk-CHmass,col-fuk-eqCH, col-gil-spr, 
gil-mir-sch} dealing with existence-uniqueness of solutions, \cite{col-fuk-diffusion, col-scar, gil-mir-sch-longtime} for studies on 
the asymptotic behaviour of solutions, and
\cite{bcst1, mir-sch, scar-VCHDBC} for analyses of
the system incorporating possibly nonlinear viscosity terms.
As far as optimal control problems are concerned,
we point out the contributions
\cite{col-far-hass-gil-spr, col-gil-spr-contr, col-gil-spr-contr2, CS17, hinter-weg}.\\

In the early $90$'s in~\cite{GL} G. Giacomin and J. Lebowitz considered the hydrodynamic limit of a microscopic model describing a $d$-dimensional lattice gas evolving via a Poisson nearest-neighbor process. In this seminal paper, the authors rigorously derived a nonlocal energy functional of the form
\begin{align} \label{eq:enfunGL}
E_{NL}(u)&=\frac{1}{4}\int_{\Omega}\int_{\Omega}K(x,y)(u(x)-u(y))^{2}\mathrm{d}
x\xd y+ \int_{\Omega}F(u(x))\xd x,
\end{align}
where $K(x,y)$ is a positive and symmetric convolution kernel, and proposed the corresponding gradient flow as a model for binary alloys undergoing phase change.

The associated evolution problem, providing a nonlocal variant of the Cahn-Hilliard PDE, is given by the following system of equations:
\begin{align}\label{eq:NLCH}
\begin{aligned}
\partial_t u+{\rm div}\, J_{NL} & =0\text{,}\\ 
J_{NL} & =-m(u)\nabla \mu_{NL}\text{,}\\
\mu_{NL} & =\frac{\delta E_{NL}(u)}{\delta u}=(K*1) u-K* u+ F'(u) \text{,}
\end{aligned}
\end{align}
where $(K*1)(x):=\int_{\Omega}K(x,y)\xd y$ and $(K*u)(x):=\int_\Omega K(x,y)u(y)\, \xd y$, for $x\in \Omega$.\\

The study of such nonlocal Cahn-Hilliard equations has recently been the subject of an intense research activity (see, e.g.
\cite{ab-bos-grass-NLCH,bat-han-NLCH,gal-gior-grass-NLCH,gal-grass-NLCH,han-NLCH} and the references therein).  All the available results in the literature dealing with nonlocal evolution of phase interfaces require the 
kernel $K$ to be symmetric and of class $W^{1,1}$.  Such requirements are usually met by checking a condition in the following form 
\begin{equation}
    \label{eq:cond-K}
    |K(x,y)|\leq C|x-y|^{-\alpha}\quad\text{with}\quad0<\alpha<\frac32
\end{equation}
(see \cite[Remark 1]{col-gil-spr-OCNLphase}).\\

The interest in this nonlocal model is motivated by its atomistic justification and its generality. A further motivation for the study of  models in the form \eqref{eq:NLCH} is the observation that, at least formally, when the interaction kernel $K$ is of the form $K(x,y)=K(|x-y|)$ and concentrates around the origin, then the behavior of the nonlocal interface evolution problems approaches that of the standard local Cahn-Hilliard equation.

This formal argument is enforced by the rigorous theory involving the variational convergence of nonlocal energies of the form \eqref{eq:enfunGL} to local integral functionals as in \eqref{eq:enfunCH}. Building upon the seminal papers by J.~Bourgain, H.~Brezis, and P.~Mironescu \cite{BBM, BBM2}, and of V. Mazy'a and T. Shaposhnikova \cite{MS, MS2}, a whole nonlocal-to-local framework has been developed for singular nonlocal kernels associated to fractional Sobolev spaces. This study has been complemented by the $\Gamma$-convergence analysis and Poincar\'e inequalities obtained by A.~C.~Ponce in \cite{ponce04, ponce}. More specifically, considering the following family of convolution kernels, identified by a small positive parameter $\varepsilon$,
\begin{equation}\label{eq:kernel}
K_\varepsilon(x,y)=\frac{\rho_\varepsilon(|x-y|)}{|x-y|^2},
\end{equation}
where $(\rho_{\ep})_{\ep}$ is a suitable sequence of mollifiers, A.~C.~Ponce showed the variational convergence
\[
\frac{1}{4}\int_{\Omega}\int_{\Omega}K_{\ep}(x,y)(u(x)-u(y))^{2}\mathrm{d}
x\xd y\to \frac12\int_{\Omega}|\nabla u(x)|^2\,\xd x.
\] 

\bigskip

The first positive result towards rendering the formal nonlocal-to-local convergence of the Cahn-Hilliard models rigorously has been achieved in~\cite{MRT18}, where the authors have focused on convergence of weak solutions of the nonlocal Cahn-Hilliard equation~\eqref{eq:NLCH} to weak solutions of its local counterpart~\eqref{eq:CH}, as the convolution kernel $K$ approximates a Dirac delta centered in the origin. In the aforementioned paper, the convergence is studied in the case of constant mobility, with a non-singular double-well potential satisfying a bounded-concavity assumption of the form
$$F''\geq -B_1,$$
for a positive constant $B_1$ small enough, (see \cite[Assumption \textbf{H3}]{MRT18}).\\

Due to the above-mentioned variational convergence result, kernels in the form \eqref{eq:kernel} are the most natural choice in the study of nonlocal phase transition problems. However, in general it is not true that these kernels enjoy a $W^{1,1}$ regularity, so that the available existence results in the literature do \emph{not} apply. In addition, the usual condition 
\eqref{eq:cond-K} is \emph{not} satisfied by $K_{\ep}$ as in \eqref{eq:kernel}. This observation renders the analysis of this class of problems very delicate and several nontrivial difficulties arise. For example, the definition and regularity of the chemical potential $\mu_{NL}$ in \eqref{eq:NLCH} relies on the properties of the linear unbounded operator
$({\bB}, D({\bf B}))$, defined as
\begin{gather*}
  D(\bB):=\{v \in L^2(\Omega):
  (K*1)v - (K*v)\in L^2(\Omega)\}\,,
  \\
  \bB(v):=(K*1)v - (K*v)\,,
  \quad \forall v\in D(\bB)\,,
\end{gather*}
whose domain $D(\bB)$ is, a priori, not explicitly characterizable and not even necessarily containing $H^1(\Omega)$ (see Subsection \ref{sec:sing-kernel}). Such endeavours are further enhanced when turning to the analysis of nonlocal diffusions driven by degenerate potentials.\\

The first contribution of this paper (see Theorem \ref{th1}) is the development of a well-posedness theory for nonlocal Cahn-Hilliard equations having singular kernels $K_\varepsilon$ (for $\ep>0$ being fixed) defined as in \eqref{eq:kernel}.\\

In our analysis, we remove the small-concavity assumption on the potential that was required in \cite{MRT18}, and include  possibly degenerate double-well potentials $F$ defined on bounded domains.
Indeed, while the classical choice for $F$ is the fourth-order polynomial 
$F_{\rm pol}(r):=\frac14(r^2-1)^2$, $r\in\mathbb{R}$, with 
minima in $\pm1$ (corresponding to the pure phases), it is well-known that, in view of the physical 
interpretation of the model, a more realistic description is given by the logarithmic 
double-well potential
$$F_{\log}(s)=\frac{\theta}{2}((1+s)\log(1+s)+(1-s)\log(1-s))+\frac{\theta_c}{2}-c s^2$$
for $0<\theta<\theta_c$ and $c>0$, which by contrast is defined on
the bounded domain $(-1,1)$ and possesses minima within the open interval $(-1,1)$. Another interesting 
example of $F$ which is included in our treatment 
is the so-called double-obstacle potential (see \cite{BE, OoP}), having the form
$$
F_{\rm ob}(s)=I_{[-1,1]}(s)+\frac12 (1-s^2),\quad 
I_{[-1,1]}(s):=
\begin{cases}
0&\text{if }s\in [-1,1]\\+\infty&\text{otherwise}\,.
\end{cases}
$$
In this latter case, the derivative $F_{\rm ob}'$ is not defined 
in the usual way, and
has to be interpreted 
as the subdifferential $\partial F_{\rm ob}$ in the sense of convex analysis (see \cite{barbu-monot}). Analogously the equations defining the chemical potential must be read as a differential inclusion instead.\\

A further extension provided by our work is to consider a nonlocal Cahn-Hilliard equation augmented by a convection term in divergence form, i.e.
\begin{align}\label{eq:NLCH-ad}
\begin{aligned}
\partial_t u+{\rm div}\, J_{NL}+ {\mathrm{div}}(\beta u)& =0\text{,}\\ J_{NL} & =-\nabla \mu_{NL}\text{,}\\
\mu_{NL} & =\frac{\delta E_{NL}(u)}{\delta u}=(K*1) u-K* u+ F'(u) \text{.}
\end{aligned}
\end{align}
Here, $\beta=\beta(t,x)$ denotes the velocity field, depending on time and space, which may be acting on the particular system in consideration. As a common choice in the literature, we considered constant mobility equal to one.\\

The interest in additional convective contributions is connected with applications in mixing and stirring of fluids, as well as in biological realizations of thin films via Langmuir-Blodgett transfer \cite{BL, Lang}. We mention in 
this direction the contributions
\cite{BDM18,col-gil-spr-CHconv,EKZ13,WORD03} on the
local Cahn-Hilliard with convection,
\cite{DPG15,DPG16,RS15} dealing with 
the nonlocal Cahn-Hilliard with local convection, and \cite{Ei87,IR07} on the nonlocal case with nonlocal convection. A nonlocal convective Cahn-Hilliard type system modelling phase-separation has been analyzed in \cite{col-gil-spr-applTych, col-gil-spr-OCNLphase}. Relevant studies in coupling the Cahn-Hilliard equation with a further equation for the velocity field have been the subject of \cite{ab-dep-gar,ab-rog,boyer,gal-grass-CHNS}.

From a mathematical viewpoint, the presence of convection terms (i.e.~when $\beta\not\equiv 0$) destroys the gradient-flow structure of the equation, causing the analysis to be even more delicate.\\

 The proof strategy for Theorem \ref{th1} relies on three main ingredients: $(1)$ a suitable approximation of the nonlinearity and an existence analysis for the approximating equations based on a fixed point argument (see Subsection \ref{subs:approx}); $(2)$ the establishment of uniform estimates by ad-hoc multiplication of the equations with suitable test functions (see Subsection \ref{s:unif_est}); $(3)$ a passage to the limit relying on nontrivial compactness and monotonicity arguments, falling outside the framework of classical Aubin-Lions embedding results (see Lemma \ref{lemma:delta-est} and Subsection \ref{subs:conv}). A delicate point is the proof of a uniform $H^1$-estimate, which strongly relies on the choice of periodic boundary conditions.\\

Our second contribution is established in Theorem \ref{th2}, where we show convergence of solutions for the nonlocal convective Cahn-Hilliard equation with singular kernel to solutions of the associated local one. Our analysis extends the work in \cite{MRT18} to a wider class of double-well potentials, satisfying no bounded-concavity assumptions and being possibly degenerate. The nonlocal-to-local convergence in Theorem \ref{th2} relies in an essential way on the uniform a-priori estimates established in the proof of Theorem \ref{th1}, and on showing the independence of the identified upper bounds from the non-locality parameter $\ep$.\\

The third and fourth main results of the paper are a regularity analysis for solutions to \eqref{eq:NLCH}. In particular, in Theorem \ref{th3} we show that, if the initial datum and the convection velocity satisfy additional integrability and differentiability assumptions, then solutions to the nonlocal Cahn-Hilliard equations exhibit further regularity. In Theorem \ref{th4} we prove that they also converge to their local counterparts in stronger topologies. The regularity analysis in Theorems \ref{th3} and \ref{th4} is the byproduct of a time-differentiation of the nonlocal Cahn-Hilliard equation, and of the use of higher-order-in-time test functions.\\

The paper is organized as follows.
Section~\ref{s:main} contains a description of the mathematical setting of the paper, the definition of weak solutions for the nonlocal and local convective Cahn-Hilliard equations, and the precise statements of the
four main results.
Sections~\ref{proof1} and \ref{proof2} are devoted to the proof of Theorems \ref{th1} and \ref{th2}, respectively. Eventually, in Section \ref{proof3} we prove Theorems \ref{th3} and \ref{th4}.

\section{Setting and main results}
\label{s:main}

\subsection{Hypotheses} Throughout the paper
we will assume the following:
\begin{description}
\item [H1] $\Omega$ is the $d$-dimensional $(d=2,3)$ flat torus and $T>0$ is a fixed final time.
\item [H2] The kernel
$K_\varepsilon:\Omega\times\Omega\to\mathbb{R}$
is defined as in~\eqref{eq:kernel}:
\[
K_\varepsilon(x,y):=
\frac{1}{|x-y|^2}
\rho_\ep(|x-y|)\,,
\qquad\text{for a.e.~}(x,y)\in\Omega\times\Omega\,
\]
where $(\rho_\ep)_{\ep>0}\subset L^1_{\rm{loc}}(0,+\infty)$
is a family of radial mollifiers on $\mathbb{R}$,
satisfying
\begin{align*}
    \rho_\ep(r)\geq 0 \qquad&\forall\,r\in \mathbb{R},\,\forall\,\ep>0\,,\\
    \operatorname{supp}(\rho_\ep)
    \subset[0,\operatorname{diam}(\Omega)]
    \qquad&\forall\,\ep>0\,,\\
    \int_0^{+\infty}\rho_\ep(r)r^{d-1}\,\xd r=
    \frac1{2 M_d}
    \qquad&\forall\,\ep>0\,,\\
    \lim_{\ep\searrow0}\int_\delta^{+\infty}
    \rho_\ep(r)r^{d-1}\,\xd r=0
    \qquad&\forall\,\delta>0\,,
\end{align*}
with $M_d:=\int_{S^{d-1}}|e_1\cdot\sigma|^2
\,d\mathcal H^{d-1}(\sigma)$.
\item[H3] $\gamma:\mathbb{R}\to2^\mathbb{R}$ is a maximal
monotone graph such that $0\in\gamma(0)$. 
This implies that $\gamma=\partial\hat\gamma$,
where $\hat\gamma:\mathbb{R}\to[0,+\infty]$ is a 
proper, convex and lower semicontinuous function. The map
$\Pi:\mathbb{R}\to\mathbb{R}$ is a Lipschitz-continuous 
function with Lipschitz constant $C_\Pi>0$.
The double-well potential $F$ will be
represented by $\hat\gamma+\hat\Pi$, where 
$\hat\Pi(t):=\int_0^t\Pi(r)\,\xd r\text{ for every }t\in \mathbb{R}$. Without restriction we will assume that $F$ is nonnegative. 
\item[H4] The velocity $\beta$ depends on space and time, and satisfies $\beta\in L^2(0,T;L^{\infty}(\Omega;\mathbb{R}^d))$.
\end{description}
We point out that all assumptions collected in $\textbf{H2}$ correspond to the requirements in \cite{ponce,ponce04}.\\

For every $\ep>0$, we consider the nonlocal Cahn-Hilliard equation with local convection
\begin{align}
    \label{eq1:NL}
    \partial_t u_\varepsilon - \Delta \mu_\varepsilon + {\mathrm{div}}(\beta u_\varepsilon)=0
    \qquad&\text{in } (0,T)\times\Omega\,,\\
    \label{eq2:NL}
    \mu_\varepsilon \in (K_\varepsilon*1)u_\varepsilon -K_\varepsilon*u_\varepsilon +\gamma(u_\varepsilon)
    +\Pi(u_\ep)
    \qquad&\text{in } (0,T)\times\Omega\,,\\
    \label{eq3:NL}
    u_\ep(0)=u_{0,\ep} \qquad&\text{in } \Omega\,,
\end{align}
and its local counterpart
\begin{align}
    \label{eq1:L}
    \partial_t u - \Delta \mu + {\mathrm{div}}(\beta u)=0
    \qquad&\text{in } (0,T)\times\Omega\,,\\
    \label{eq2:L}
    \mu \in -\Delta u +\gamma(u)
    +\Pi(u)
    \qquad&\text{in } (0,T)\times\Omega\,,\\
    \label{eq3:L}
    u(0)=u_{0} \qquad&\text{in } \Omega\,.
\end{align}

\subsection{Notation, preliminaries and comments}
\label{sec:sing-kernel}
In the sequel we will identify $L^2(\Omega)$
with its dual, so that $(H^1(\Omega), L^2(\Omega), (H^1(\Omega))^*)$ will be a classical 
Hilbert triplet. We will use the symbol $(v)_\Omega$ for 
$\frac1{|\Omega|}
\langle v,1\rangle_{(H^1(\Omega))^*, H^1(\Omega)}$
for every $v\in (H^1(\Omega))^*$. Note that for $v\in L^2(\Omega)$, $(v)_{\Omega}$ coincides with the usual average.
We recall that the operator
\[
(-\Delta)^{-1}:\{v\in (H^1(\Omega))^*:(v)_\Omega=0\}
\rightarrow \{w\in H^1(\Omega): (w)_{\Omega}=0\}
\]
is defined
as the map assigning to every $v\in (H^1(\Omega))^*$
with null mean 
the unique element $w\in H^1(\Omega)$ such that
\[
  (w)_\Omega=0\,, \qquad\text{and}\qquad 
  \int_\Omega\nabla w(x)\cdot\nabla\varphi(x)\, \xd x  = 
  \langle v, \varphi\rangle_{(H^1(\Omega))^*,H^1(\Omega)}
  \quad\forall\,\varphi\in H^1(\Omega)\,.
\]
It is well known that $(-\Delta)^{-1}$ is a 
linear isomorphism.

In this paper $C$ indicates a generic positive constant, 
possibly varying from line to line, depending only on 
the setting \textbf{H1}--\textbf{H4}. The  dependence of constants on a specific parameter will be indicated explicitly through a subscript.

\medskip

We collect here some useful properties of 
the nonlocal term.
We define the operator
$({\bB_\ep}, D({\bf B}_\ep))$
on $L^2(\Omega)$ in the following way:
\begin{gather*}
  D(\bB_\ep):=\{v \in L^2(\Omega):
  (K_\ep*1)v - (K_\ep*v)\in L^2(\Omega)\}\,,
  \\
  \bB_\ep(v):=(K_\ep*1)v - (K_\ep*v)\,,
  \quad \forall v\in D(\bB_\ep)\,.
\end{gather*}
It is clear that $\bB_\ep$ is a linear unbounded 
operator in $L^2(\Omega)$, and that for every 
$v\in D(\bB_\ep)$ we have the representation
\[
  \bB_\ep(v)(x)=
  \int_\Omega 
  \rho_\ep(|x-y|)
  \frac{v(x)-v(y)}{|x-y|^2}\,\xd y
  \quad\text{for a.e.~}x\in\Omega\,.
\]
We point out that the domain $D(\bB_\ep)$
is non-trivial.
More specifically, we have the following result.
\begin{lemma}
  \label{lem:dom_B}
  For every $\ep,\sigma>0$, there holds
  $C^{0,\sigma}(\overline\Omega)\subset
  D({\bf B}_\ep)$. Additionally,  
  there exists a constant
  $C_{\ep,\sigma}>0$ (only dependent on $\ep$ and $\sigma$) such that
  \begin{equation}
      \label{eq:dom_B1}
      \|{\bf B}_\ep(v)\|_{L^2(\Omega)}
      \leq C_{\ep,\sigma}\|v\|_{C^{0,\sigma}(\overline\Omega)} 
      \quad\forall\,v\in
      C^{0,\sigma}(\overline\Omega)\,.
  \end{equation}
  In particular, for every $s>\frac32$,
  $H^{s}(\Omega)\subset D({\bf B}_\ep)$ and
  there exists $C_{\ep,s}>0$ such that
  \begin{equation}
      \label{eq:dom_B2}
      \|{\bf B}_\ep(v)\|_{L^2(\Omega)}
      \leq C_{\ep,s}\|v\|_{H^s(\Omega)}
      \quad\forall\,v\in
      H^s(\Omega)\,.
  \end{equation}
\end{lemma}
\begin{proof}
  A direct computation shows that
  for every $v\in C^{0,\sigma}(\overline\Omega)$
  and for almost every $x\in\Omega$,
  \[
  |{\bf B}_\ep(v)(x)|\leq
  \int_\Omega 
  \rho_\ep(|x-y|)
  \frac{|v(x)-v(y)|}{|x-y|^2}\,\xd y\leq
  \|v\|_{C^{0,\sigma}(\overline\Omega)}
  \int_\Omega\frac{\rho_\ep(|x-y|)}
  {|x-y|^{2-\sigma}}\,\xd y,
  \]
  where
  \[
  C_{\ep,\sigma}:=\int_\Omega\frac{\rho_\ep(|x-y|)}
  {|x-y|^{2-\sigma}}\,\xd y <+\infty
  \]
  thanks to \textbf{H2}. The second part 
  of the Lemma follows by the Sobolev embedding $H^{s}(\Omega)\hookrightarrow
  C^{0,\sigma}(\overline\Omega)$
  for every $s>\frac32$ and $\sigma\in(0,s-\frac32)$.
\end{proof}
The operator $({\bf B}_\ep, D({\bf B_\ep}))$
has been defined as a linear unbounded operator
on $L^2(\Omega)$. Note that it is not 
necessarily true that $H^1(\Omega)
\subset D({\bf B_\ep})$.
Nevertheless,
we now show that actually
$({\bf B}_\ep, D({\bf B_\ep}))$ can be extended, uniformly in $\ep$,
to a linear bounded operator from $H^1(\Omega)$
to its dual.
\begin{lemma}
  \label{lem:B_H1}
  For every $\ep>0$ the operator 
  $(D({\bf B_\ep}), {\bf B}_\ep)$
  can be uniquely extended to a linear continuous operator 
  ${\bf B}_\ep:H^1(\Omega)\to (H^1(\Omega))^*$. Additionally, there exists a positive constant 
  $C$, independent of $\ep$, such that 
  \begin{align}\label{estB} 
  \|\bB_\ep(v)\|_{(H^1(\Omega))^*}\leq 
  C\|\nabla v\|_{L^2(\Omega)} 
  \qquad\forall\,v\in H^1(\Omega)\,.
  \end{align}
  In particular, the family $({\bf B}_\ep)_\ep$
  is uniformly bounded in $\mathscr{L}(H^1(\Omega), (H^1(\Omega))^*)$
  and there exists ${\bf B}\in \mathscr{L}(H^1(\Omega), (H^1(\Omega))^*)$
  and an infinitesimal 
  sequence $(\ep_n)_n$
  such that 
  \[
  \lim_{n\to\infty}\langle{\bf B}_{\ep_n}(v),\psi\rangle_{(H^1(\Omega))^*,H^1(\Omega)}=
  \langle{\bf B}(v),\psi\rangle_{(H^1(\Omega))^*,H^1(\Omega)}
  \qquad\,\forall\,v,\psi\in H^1(\Omega)\,.
  \]
\end{lemma}
\begin{proof}
By the H\"older inequality and \cite[Theorem 1]{BBM}, we infer that 
\begin{align*}
&\langle {\bf B}_\varepsilon(v),\psi\rangle_{(H^{1}(\Omega))^*, H^1(\Omega)}=\frac{1}{2}\int_\Omega\int_\Omega K_\varepsilon(x,y)(v(x)-v(y))(\psi(x)-\psi(y))\xd y \xd x\\
&\leq\frac{1}{2}\Bigl(\int_\Omega\int_\Omega K_\varepsilon(x,y) |v(x)-v(y)|^2 \xd y \xd x\Bigr)^{1/2}\Bigl(\int_\Omega\int_\Omega K_\varepsilon(x,y) |\psi(x)-\psi(y)|^2\xd y \xd x \Bigr)^{1/2}\\
&\leq C\Vert \nabla v\Vert_{L^2(\Omega)}\Vert \nabla \psi\Vert_{L^2(\Omega)}
\end{align*}
for every $v,\psi\in H^1(\Omega)$. 
This implies that $(D({\bf B_\ep}), {\bf B}_\ep)$ can be extended uniquely as required
(the uniqueness follows Lemma \ref{lem:dom_B}, and from the density of $C^{0,\sigma}(\overline{\Omega})$ in
$D({\bf B_\ep})$).
The second part of the lemma follows by observing that \eqref{estB}
implies the uniform boundedness of $({\bf B}_\ep)_\ep$ in $\mathscr{L}(H^1(\Omega), (H^1(\Omega))^*)$, and hence its precompactness in the weak operator
topology of $\mathscr{L}(H^1(\Omega), (H^1(\Omega))^*)$.
\end{proof}

In what follows, a crucial role is also played by the nonlocal energy contribution
\begin{align*}
E_{\ep}(v):=\frac14\int_{\Omega}\int_{\Omega}K_{\ep}(x,y)(v(x)-v(y))^2\xd y\xd x\,\qquad
\forall v\in H^1(\Omega)\,.
\end{align*}
Owing to \cite[Theorem 1]{BBM},
we have that $E_\ep$
is well-defined, convex, and its differential $D E_\ep
  :H^1(\Omega)\to (H^1(\Omega))^*$ is given by
\[
  D E_\ep={\bf B}_\ep.
\]

Moreover, by \cite{BBM} the asymptotic behavior of $E_{\ep}$ as $\ep\to 0^+$ can be characterized as follows 
\begin{equation}
    \label{eq:BBM}
  \lim_{\ep\to0^+}E_\ep(v)=\frac12\int_\Omega|\nabla v(x)|^2\xd x \quad\forall\,v\in H^1(\Omega)\,.
\end{equation}
As a corollary, we deduce the following identification of the operator ${\bf B}$ in Lemma~\ref{lem:B_H1_2}.

\begin{lemma}\label{lem:B_H1_2}
Let $(D(\bf B_{\ep}),{\bf B}_{\ep})_{\ep}$ and ${\bf B}$ be as in Lemma~\ref{lem:B_H1}. Then,
\[
\lim_{\ep\to 0}\langle{\bf B}_{\ep}(v),\psi\rangle_{(H^1(\Omega))^*,H^1(\Omega)}=\langle-\Delta v,\psi\rangle_{(H^1(\Omega))^*, H^1(\Omega)},
\]
where
\[
  \langle-\Delta v,\psi\rangle_{(H^1(\Omega))^*, H^1(\Omega)}:=\int_\Omega\nabla v(x)\cdot \nabla\psi(x)\,\xd x \quad \forall v,\psi\in H^1(\Omega)\,.
\]
\end{lemma}

\begin{proof}
By the characterization of the differential of $E_{\ep}$, we have that 
\[
  E_{\ep}(v_1) + \langle{\bf B}_{\ep}(v_1),v_2-v_1
  \rangle_{(H^1(\Omega))^*, H^1(\Omega)}
  \leq E_{\ep}(v_2)
\]
for every $v_1,v_2\in H^1(\Omega)$. Hence, for every subsequence $(\ep_n)_n$ as in Lemma \ref{lem:B_H1},
letting $n\to \infty$, by \eqref{eq:BBM} we conclude that
\[
  \frac12\int_\Omega|\nabla v_1(x)|^2\xd x+
  \langle{\bf B}(v_1),v_2-v_1
  \rangle_{(H^1(\Omega))^*, H^1(\Omega)}
  \leq \frac12\int_\Omega|\nabla v_2(x)|^2\xd x,
\]
from which ${\bf{B}}=-\Delta$.
In particular, this implies that the convergence holds along the entire sequence $\ep$.
\end{proof}

We conclude this subsection with a lemma providing
two fundamental compactness inequalities 
involving the family of operators $({\bf B}_\ep)_\ep$. Such results are nontrivial,
since they do not fall in the classical 
framework of the Aubin-Lions lemmas.
The next lemma is
a uniform counterpart to \cite[Lemma 1]{MRT18}.
\begin{lemma}
\label{lemma:delta-est}
For every $\delta>0$ there exist constants $C_{\delta}>0$ and $\ep_\delta>0$ 
with the following properties: 
\begin{enumerate}
    \item 
For every sequence $(f_\ep)_\ep\subset H^1(\Omega)$ there holds
\begin{align}
   \notag
    \|f_{\ep_1}-f_{\ep_2}\|^2_{H^1(\Omega)}&\leq 
    \delta 
   \int_{\Omega}\int_{\Omega}
   K_{\ep_1}(x,y)|\nabla f_{\ep_1}(x)-
  \nabla f_{\ep_1}(y)|^2\xd y\xd x\\
  \label{eq:comp1}
  &+
  \delta 
   \int_{\Omega}\int_{\Omega}
   K_{\ep_2}(x,y)|\nabla f_{\ep_2}(x)-
  \nabla f_{\ep_2}(y)|^2\xd y\xd x
  +C_{\delta}\|f_{\ep_1}-f_{\ep_2}\|^2_{L^2(\Omega)}
\end{align}
for every $0<\ep_1,\ep_2<\ep_\delta$.
\item For every sequence $(f_\ep)_\ep\subset L^2(\Omega)$ there holds
\begin{align}
   \label{eq:comp2}
    \|f_{\ep_1}-f_{\ep_2}\|^2_{L^2(\Omega)}&\leq 
    \delta 
   E_{\ep_1}(f_{\ep_1})
   +\delta 
   E_{\ep_2}(f_{\ep_2})
  +C_{\delta}
  \|f_{\ep_1}-f_{\ep_2}\|^2_{(H^1(\Omega))^*}
\end{align}
for every $0<\ep_1,\ep_2<\ep_\delta$.
\end{enumerate}
\end{lemma}
\begin{proof}
Assume by contradiction that \eqref{eq:comp1} is false. Then, there exists $\bar{\delta}>0$ having the following property: for every $n\in \mathbb{N}$ we can find a sequence $(f_{\ep}^n)_{\ep}\subset H^1(\Omega)$ and two parameters $\ep_n^1,\ep_n^2<\frac1n$ such that
\begin{align*}
  \|f^n_{\ep_n^1}-f_{\ep_n^2}^n\|^2_{H^1(\Omega)}&> 
 \bar\delta  \int_{\Omega}\int_{\Omega}
  K_{\ep_n^1}(x,y)|\nabla f_{\ep_n^1}^n(x)-
    \nabla f_{\ep_n^1}^n(y)|^2\xd y\xd x\\
  &+
  \bar\delta  \int_{\Omega}\int_{\Omega}
  K_{\ep_n^2}(x,y)|\nabla f_{\ep_n^2}^n(x)-
    \nabla f_{\ep_n^2}^n(y)|^2\xd y\xd x
   +n\|f_{\ep_n^1}^n-f_{\ep_n^2}^n\|^2_{L^2(\Omega)}\,.
\end{align*}
Noting that $\|f_{\ep_n^1}^n
-f^n_{\ep_n^2}\|_{H^1(\Omega)}>0$
for every $n$ and setting
\[
g_n^1:=
\frac{f^n_{\ep_n^1}}
{\|f^n_{\ep_n^1}-f^n_{\ep_n^2}\|_{H^1(\Omega)}}\,,
\qquad
g_n^2:=
\frac{f^n_{\ep_n^2}}
{\|f^n_{\ep_n^1}-f^n_{\ep_n^2}\|_{H^1(\Omega)}}\,,
\]
we have
\begin{align*}
 &\bar\delta
    \int_{\Omega}\int_{\Omega}
   K_{\ep_n^1}(x,y)|\nabla g_n^1(x)-
    \nabla g_n^1(y)|^2\xd y\xd x\\
    &+
    \bar\delta
    \int_{\Omega}\int_{\Omega}
   K_{\ep_n^2}(x,y)|\nabla g_n^2(x)-
    \nabla g_n^2(y)|^2\xd y\xd x
    +n\|g_n^1-g_n^2\|^2_{L^2(\Omega)}<1 
   \qquad\forall\,n\in\mathbb{N}\,.
\end{align*}
Hence, $g_n^1-g_n^2\to0$ strongly in $L^2(\Omega)$
and the families $(\nabla g_n^1)_n$
and $(\nabla g_n^2)_n$
are relatively strongly compact in $L^2(\Omega;\mathbb{R}^d)$ by \cite[Theorem~1.2]{ponce04}.
We deduce that $g_n^1-g_n^2\to0$ strongly in
$H^1(\Omega)$, but this is a contradiction
since by definition we have 
$\|g_n^1-g_n^2\|_{H^1(\Omega)}=1$ for all $n$. The argument for \eqref{eq:comp2} is entirely analogous.
\end{proof}

\subsection{Main results}
Before stating our main results, let us recall the notion of weak solutions to both the nonlocal as well as the local Cahn-Hilliard equation with local convection.

\begin{definition}[Solution to the nonlocal Cahn-Hilliard equation]\label{def:wsnl-ad}
Let $\varepsilon >0 $ and $T>0$ be fixed. 
A solution to the nonlocal Cahn-Hilliard 
equation~\eqref{eq1:NL}--\eqref{eq3:NL} 
on $[0,T]$, and associated with the initial datum 
$u_{0,\varepsilon}\in L^2(\Omega)$, 
is a triplet $(u_\ep,\mu_\ep, \xi_\ep)$ with the following properties 
\begin{gather*}
u_\varepsilon \in H^1(0,T;(H^1(\Omega))^\ast) \cap L^2(0,T;H^1(\Omega))\,,\\
\mu_\ep \in L^2(0,T; H^1(\Omega))\,,\qquad
\xi_\ep \in L^2(0,T; L^2(\Omega))\,,\\
\mu_\ep={\bf B}_\ep(u_\ep) +\xi_\ep + \Pi(u_\ep)\,,
\qquad \xi_\ep \in \gamma (u_\ep)\,\quad\text{almost everywhere in }
(0,T)\times \Omega,
\end{gather*}
satisfying $u_\varepsilon (0) = u_{0,\varepsilon}$, and such that
\begin{gather}\label{weak_NL-ad}
\langle\partial_t u_\varepsilon(t), \varphi\rangle_{(H^1(\Omega))^*,H^1(\Omega)} + \int_\Omega \nabla\mu_\ep(t,x)\cdot \nabla \varphi(x)\, \xd x =\int_\Omega \beta(t,x) u_\ep(t,x)\cdot \nabla\varphi(x)\,\xd x
\end{gather}
for all $\varphi \in H^1(\Omega)$, and for
almost every $t\in(0,T)$.
\end{definition}

\begin{definition}[Solution to the local Cahn-Hilliard equation]\label{def:wsl-ad}
Let $T>0$ be fixed. 
A solution to the local Cahn-Hilliard equation
\eqref{eq1:L}--\eqref{eq3:L}
on $[0,T]$, and associated with the initial datum 
$u_{0}\in H^1(\Omega)$, is a triplet $(u,\mu,\xi)$ 
with the following properties 
\begin{gather*}
u \in H^1(0,T;(H^1(\Omega))^\ast) \cap L^2(0,T;H^2(\Omega))\,,\\
\mu \in L^2(0,T; H^1(\Omega))\,,\qquad
\xi \in L^2(0,T; L^2(\Omega))\,,\\
\mu=-\Delta u +\xi + \Pi(u)\,, \qquad\xi\in\gamma(u)\,\quad\text{almost everywhere in}\,(0,T)\times \Omega,
\end{gather*}
satisfying $u(0) = u_{0}$, and such that
\begin{gather}\label{weak-ad}
\langle\partial_t u(t), \varphi\rangle_{(H^1(\Omega))^*,H^1(\Omega)} + \int_\Omega \nabla\mu(t,x)\cdot \nabla \varphi(x)\, \xd x =\int_\Omega \beta(t,x) u(t,x)\cdot \nabla\varphi(x)\,\xd x
\end{gather}
for all $\varphi \in H^1(\Omega)$, for
almost every $t\in (0,T)$.
\end{definition}

Our first result is the well-posedness of solutions to the nonlocal Cahn-Hilliard equation.
\begin{theorem}
 \label{th1}
  Let assumptions \emph{\textbf{H1--H4}} be satisfied, and
  for every $\ep>0$ let
  \begin{equation}\label{ip_u0_ep}
      u_{0,\ep} \in L^2(\Omega)\,, \qquad
      \hat\gamma(u_{0,\ep})\in L^1(\Omega)\,, \qquad
      E_\ep(u_{0,\ep})<+\infty\,, \qquad
      (u_{0,\ep})_\Omega \in \operatorname{Int}D(\gamma)\,.
  \end{equation}
  Then, there exists $\ep_0>0$ having the following property: for every $\ep<\ep_0$ there exists a unique 
  solution $(u_\ep, \mu_\ep, \xi_\ep)$ to \eqref{eq1:NL}--\eqref{eq3:NL}
  associated with the initial datum $u_{0,\ep}$, according to Definition \ref{def:wsnl-ad}.
  Furthermore, if $(\beta^1,u_{0,\ep}^1)$ and $(\beta^2,
  u_{0,\ep}^2)$ are two sets of data satisfying
  \emph{\textbf{H4}} and \eqref{ip_u0_ep}, with 
  $(u_{0,\ep}^1)_\Omega=(u_{0,\ep}^2)_\Omega$, then
  there exists a positive constant $M_\ep$,
  depending only on the setting
  \emph{\textbf{H1}}--\emph{\textbf{H3}} and on the norms of the  data $(\beta^1,u_{0,\ep}^1)$ and $(\beta^2,
  u_{0,\ep}^2)$ appearing in \emph{\textbf{H4}}
  and \eqref{ip_u0_ep}, such that,
  for any respective solution
  $(u_\ep^1,\mu_\ep^1,\xi_\ep^1)$ and 
  $(u_\ep^2,\mu_\ep^2,\xi_\ep^2)$
  to the nonlocal equation \eqref{eq1:NL}--\eqref{eq3:NL},
  \begin{align*}
  \|u_\ep^1-u_\ep^2\|_{C^0([0,T]; (H^1(\Omega))^*)}^2
  &+\|E_\ep(u_\ep^1-u_\ep^2)\|_{L^1(0,T)}\\
  &\leq M_\ep
  \left(
  \|u_{0,\ep}^1-u_{0,\ep}^2\|_{(H^1(\Omega))^*}^2
  +\|\beta^1-\beta^2\|_{L^2(0,T; L^3(\Omega))}^2\right)\,.
  \end{align*}
\end{theorem}
The second result concerns nonlocal-to-local convergence.
\begin{theorem}\label{th2}
Let assumptions \emph{\textbf{H1--H4}} be satisfied. Let 
$u_0\in H^1(\Omega)$, and
for every $\ep>0$ let $u_{0,\ep}$ satisfy \eqref{ip_u0_ep} and be such that
\begin{gather}
    \label{ip_u0_ep_unif}
    \sup_{\ep\in(0,\ep_0)}\left(
    \|u_{0,\ep}\|_{L^2(\Omega)}^2 + \|\hat\gamma(u_{0,\ep})\|_{L^1(\Omega)}
    +E_\ep(u_{0,\ep})\right) <+\infty\,,\\
    \label{ip_u0_ep_unif2}
    \exists\,[a_0,b_0]\subset\operatorname{Int}D(\gamma):\quad
    a_0\leq(u_{0,\ep})_\Omega\leq b_0 \quad\forall\,\ep\in(0,\ep_0)\,,\\
    \label{ip_u0_ep_unif3}
    u_{0,\ep}\rightharpoonup u_0 \quad\text{in } L^2(\Omega) 
    \quad\text{as } \ep\to0^+\,.
\end{gather}
Let $(u_\ep,\mu_\ep,\xi_\ep)$
be the unique solution to 
\eqref{eq1:NL}--\eqref{eq3:NL} 
associated to $u_{0,\ep}$ given by Theorem~\ref{th1},
and let $(u,\mu,\xi)$ be the unique solution to the local equation \eqref{eq1:L}--\eqref{eq3:L}
associated to $u_0$, according to Definition \ref{def:wsl-ad}.

Then, as $\ep\searrow0$,
\begin{align*}
    u_\ep \to u \qquad&\text{strongly in } 
    C^0([0,T]; L^2(\Omega))\cap L^2(0,T; H^1(\Omega))\,,\\ 
    \partial_t u_\ep \rightharpoonup \partial_t u \qquad&\text{weakly* in }
    L^2(0,T;(H^1(\Omega))^*)\,,\\
    \mu_\ep \rightharpoonup \mu \qquad&\text{weakly in }
    L^2(0,T; H^1(\Omega))\,,\\
    \xi_\ep \rightharpoonup \xi \qquad&\text{weakly in }
    L^2(0,T; L^2(\Omega))\,.
\end{align*}
\end{theorem}

The last two results that we present deal with regularity 
of solutions to the nonlocal equation. In particular, we show
that if the data are more regular, then the 
solution to the nonlocal equation inherits a further 
regularity, and the convergences to the local equation
are obtained in stronger topologies.

\begin{theorem}
  \label{th3}
  Let assumptions \emph{\textbf{H1}--\textbf{H4}} be satisfied,
  and suppose also that
  \begin{equation}
      \label{beta_reg}
      \beta\in H^1(0,T; L^3(\Omega;\mathbb{R}^d))\,.
  \end{equation}
  For every $0<\ep<\ep_0$ let $u_{0,\ep}$ satisfy \eqref{ip_u0_ep} and
  \begin{equation}
      \label{ip_u0_ep_reg}
      u_{0,\ep}\in L^6(\Omega)\,,\qquad
      {\bf B}_\ep(u_{0,\ep}) + \xi_{0,\ep}+\Pi(u_{0,\ep}) \in H^1(\Omega) 
      \quad\forall\,\xi_{0,\ep}\in\gamma(u_{0,\ep})\,.
  \end{equation}
  Then the unique solution $(u_\ep,\mu_\ep, \xi_\ep)$
  to the nonlocal equation \eqref{eq1:NL}--\eqref{eq3:NL}
  with respect to the initial datum $u_{0,\ep}$ also satisfies
  \[
      u_\ep \in W^{1,\infty}(0,T; (H^1(\Omega))^*)\cap
      H^1(0,T; L^2(\Omega))\cap L^2(0,T; H^1(\Omega))\,.
  \]
  If also
  \begin{equation}
      \label{beta_reg2}
      \beta \in L^\infty(0,T; L^\infty(\Omega;\mathbb{R}^d))\,,
  \end{equation}
  then in addition
  \[
  \mu_\ep \in L^\infty(0,T; H^1(\Omega))\,,\qquad
      \xi_\ep \in L^\infty(0,T; L^2(\Omega))\,.
  \]
  If also
  \begin{equation}
      \label{beta_reg3}
      \operatorname{div}\beta \in L^\infty(0,T; L^3(\Omega))\,,
  \end{equation}
  then in addition
  \[
  \mu_\ep \in L^2(0,T; H^2(\Omega))\,.
  \]
\end{theorem}

\begin{theorem}
  \label{th4}
  Let assumptions \emph{\textbf{H1--H4}} be satisfied. 
  Let $u_0\in H^1(\Omega)$, and
  for every $\ep>0$ let $u_{0,\ep}$ satisfy \eqref{ip_u0_ep}, \eqref{ip_u0_ep_unif}--\eqref{ip_u0_ep_unif3}, 
  \eqref{ip_u0_ep_reg}
  and
  \begin{equation}
      \label{ip_u0_ep_reg_unif}
      \sup_{\ep\in(0,\ep_0),\; \xi_{0,\ep}\in\gamma(u_{0,\ep})}
      \left(\|u_{0,\ep}\|_{L^6(\Omega)} +
      \|{\bf B}_\ep(u_{0,\ep}) + \xi_{0,\ep}+\Pi(u_{0,\ep}\|_{H^1(\Omega)}
      \right)<+\infty\,.
  \end{equation}
  Denoting by $(u,\mu,\xi)$ the unique solution to the local equation \eqref{eq1:L}--\eqref{eq3:L}, 
  if \eqref{beta_reg} holds
  then,
  in addition to the convergences in Theorem~\ref{th2},
  \begin{align*} 
    u_\ep \rightharpoonup u \qquad&\text{weakly* in }
    W^{1,\infty}(0,T;(H^1(\Omega))^*)\cap H^1(0,T; L^2(\Omega))\,.
  \end{align*}
  If also \eqref{beta_reg2} holds, then
  \begin{align*}
    \mu_\ep \rightharpoonup \mu \qquad&\text{weakly* in }
    L^\infty(0,T; H^1(\Omega))\,,\\
    \xi_\ep \rightharpoonup \xi \qquad&\text{weakly* in }
    L^\infty(0,T; L^2(\Omega))\,.
  \end{align*}
  If also \eqref{beta_reg3} holds, then
  \[
  \mu_\ep \rightharpoonup \mu \qquad\text{weakly in }
    L^2(0,T; H^2(\Omega))\,.
  \]
\end{theorem}

\section{Proof of Theorem~\ref{th1}}
\label{proof1}

This section contains the proof of existence 
of a solution $(u_\ep,\mu_\ep,\xi_\ep)$
to the nonlocal convective Cahn-Hilliard equation.
We subdivide it in different steps.
In this section, $\ep>0$ is fixed.

\subsection{Approximation}
\label{subs:approx}
For every $\lambda>0$, let $\gamma_\lambda:\mathbb{R}\to
\mathbb{R}$ be the Yosida approximation of $\gamma$, having Lipschitz constant $1/\lambda$,
and set 
$\hat\gamma_\lambda(s):=\int_0^s\gamma_\lambda(r)\,\xd r$ for every $s\in \mathbb{R}$.
We consider the approximated problem
\begin{align}
    \label{1_app}
    \partial_t u_\ep^\lambda - \Delta \mu_\ep^\lambda
    +\operatorname{div}(\beta_\lambda u_\ep^\lambda) = 0
    \qquad&\text{in } (0,T)\times\Omega\,,\\
    \label{2_app}
    \mu_\ep^\lambda = -\lambda\Delta u_\ep^\lambda
    +{\bf B}_\ep(u_\ep^\lambda)
    +\gamma_\lambda(u_\ep^\lambda)
    +\Pi(u_\ep^\lambda)
    \qquad&\text{in } (0,T)\times\Omega\,,\\
    \label{3_app}
    u_\ep^\lambda(0)=u_{0,\ep}^\lambda \qquad&\text{in } \Omega\,,
\end{align}
where $\beta_\lambda:=P_\lambda \beta$,
$P_\lambda:\mathbb{R}^d\to\mathbb{R}^d$ is the projection
on the closed ball of radius $\frac1\lambda$, and the 
initial datum $u_{0,\ep}^\lambda$ satisfies
\begin{gather}
    \label{ip_u0_ep_lam}
    u_{0,\ep}^\lambda \in H^1(\Omega)\,, \qquad
    u_{0,\ep}^\lambda\to u_{0,\ep} \quad\text{in } L^2(\Omega)\,,\\
    \label{ip_u0_ep_lam2}
    \sup_{\lambda\in(0,\lambda_0)}\left(
    \lambda\|u_{0,\ep}^\lambda\|_{H^1(\Omega)}^2+
    \|\hat\gamma_\lambda(u_{0,\ep}^\lambda)\|_{L^1(\Omega)}
    +E_\ep(u_{0,\ep}^\lambda)\right)<+\infty
\end{gather}
for a certain $\lambda_0>0$
(possibly depending on $\ep$).

\begin{remark}
The existence of an approximating sequence $(u_{0,\ep}^\lambda)_\lambda$
satisfying \eqref{ip_u0_ep_lam}--\eqref{ip_u0_ep_lam2} is 
guaranteed by \eqref{ip_u0_ep}.
For example,
let us consider
the classical 
elliptic regularization given by the unique solution to 
the problem
\begin{equation}\label{elliptic}
  u_{0,\ep}^\lambda - \lambda\Delta u_{0,\ep}^\lambda = u_{0,\ep} \qquad\text{in } \Omega\,.
\end{equation}
Note that we have not specified any boundary 
conditions for $u_{0,\ep}^\lambda$
as we are working on the torus $\Omega$
(hence we have implicitly required 
periodic boundary conditions for $u_{0,\ep}^\lambda$).
Let us show that 
\eqref{ip_u0_ep_lam}--\eqref{ip_u0_ep_lam2}
are satisfied by this choice.
Testing \eqref{elliptic}
by $u_{0,\ep}^\lambda$ and using the 
Young inequality on the right-hand side
we obtain
\[
  \frac12\|u_{0,\ep}^\lambda\|_{L^2(\Omega)}^2
  +\lambda\|\nabla u_{0,\ep}^\lambda\|_{L^2(\Omega)}^2
  \leq \frac12\|u_{0,\ep}\|_{L^2(\Omega)}^2\,.
\]
This readily implies \eqref{ip_u0_ep_lam}
and the first bound in \eqref{ip_u0_ep_lam2}.
Moreover, testing \eqref{elliptic}
by $\gamma_\lambda(u_{0,\ep}^\lambda)$ we get
\[
  \int_\Omega\gamma_\lambda(u_{0,\ep}^\lambda(x))
  u_{0,\ep}^\lambda(x)\,\xd x
  +\lambda\int_\Omega\gamma_\lambda'(u_{0,\ep}^\lambda(x))
  |\nabla u_{0,\ep}^\lambda(x)|^2\,\xd x
  =\int_\Omega\gamma_\lambda(u_{0,\ep}^\lambda(x))
  u_{0,\ep}(x)\,\xd x\,.
\]
Denoting by $\hat\gamma_\lambda^*$ the convex
conjugate of $\hat\gamma_\lambda$,
the first term on the left-hand side reads as
\[
  \int_\Omega\gamma_\lambda(u_{0,\ep}^\lambda(x))
  u_{0,\ep}^\lambda(x)\,\xd x =
  \int_\Omega
  \hat\gamma_\lambda(u_{0,\ep}^\lambda(x))\,\xd x
  +\int_\Omega\hat\gamma_\lambda^*
  (\gamma_\lambda(u_{0,\ep}^\lambda(x)))\,\xd x\,,
\]
the second term on the left-hand side is nonnegative
by the monotonicity of $\gamma_\lambda$, while
the right-hand side can be bounded
through the Young inequality as
\begin{align*}
\int_\Omega\gamma_\lambda(u_{0,\ep}^\lambda(x))
  u_{0,\ep}(x)\,\xd x
  &\leq 
  \int_\Omega
  \hat\gamma_\lambda(u_{0,\ep}(x))\,\xd x
  +\int_\Omega\hat\gamma_\lambda^*
  (\gamma_\lambda(u_{0,\ep}^\lambda(x)))\,\xd x\\
  &\leq 
  \int_\Omega
  \hat\gamma(u_{0,\ep}(x))\,\xd x
  +\int_\Omega\hat\gamma_\lambda^*
  (\gamma_\lambda(u_{0,\ep}^\lambda(x)))\,\xd x\,.
\end{align*}
Rearranging the terms we get 
$\|\hat\gamma_\lambda(u_{0,\ep}^\lambda)\|_{L^1(\Omega)}
\leq\|\hat\gamma(u_{0,\ep})\|_{L^1(\Omega)}$,
from which the second bound in \eqref{ip_u0_ep_lam2}.
Finally,
testing \eqref{elliptic}
by $\bB_\ep(u_{0,\ep}^\lambda)$ 
we have
\[
  \int_\Omega \bB_\ep(u_{0,\ep}^\lambda)(x)
  u_{0,\ep}^\lambda(x)\,\xd x
  +\lambda\int_\Omega
  \nabla\bB_\ep(u_{0,\ep}^\lambda)(x)
  \cdot\nabla u_{0,\ep}^\lambda(x)\,\xd x=
  \int_\Omega u_{0,\ep}(x)
  \bB_\ep(u_{0,\ep}^\lambda)(x)\,\xd x\,,
\]
where, thanks to the periodic boundary conditions,
on the left-hand side we have
\[
\int_\Omega \bB_\ep(u_{0,\ep}^\lambda)(x)
u_{0,\ep}^\lambda(x)\,\xd x
=2 E_\ep(u_{0,\ep}^\lambda)
\]
and
\[
\lambda\int_\Omega\nabla\bB_\ep(u_{0,\ep}^\lambda)(x)
  \cdot\nabla u_{0,\ep}^\lambda(x)\,\xd x=
  \frac\lambda2\int_\Omega\int_\Omega
  K_\ep(x,y)|\nabla u_{0,\ep}^\lambda(x)
  -\nabla u_{0,\ep}^\lambda(y)|^2\,\xd x\,\xd y\,.
\]
On the right-hand side, 
by the H\"older and Young inequalities,
we have
\begin{align*}
\int_\Omega u_{0,\ep}(x)
\bB_\ep(u_{0,\ep}^\lambda)(x)\,\xd x&=
\frac12\int_\Omega\int_\Omega
K_\ep(x,y)
(u_{0,\ep}(x)-u_{0,\ep}(y))
(u_{0,\ep}^\lambda(x)-u_{0,\ep}^\lambda(y))
\,\xd x\,\xd y\\
&\leq\sqrt{2E_\ep(u_{0,\ep})}
\sqrt{2E_\ep(u_{0,\ep}^\lambda)}\leq
E_\ep(u_{0,\ep}) + E_\ep(u_{0,\ep}^\lambda)\,.
\end{align*}
Rearranging the terms we get 
$E_\ep(u_{0,\ep}^\lambda)\leq E_\ep(u_{0,\ep})$,
from which the third bound in \eqref{ip_u0_ep_lam2}.
\end{remark}

In this subsection, we show existence of an approximated solution $(u_\ep^\lambda,\mu_\ep^\lambda)$ for every $\lambda>0$ fixed. The proof strategy relies on the use of a fixed-point argument.
\\

For every $w \in L^2(0,T; H^s(\Omega))$ with $s\in \left(\frac{3}{2},2\right)$,
Lemma \ref{lem:dom_B} ensures that
\[
  {\bf B}_\ep(w) \in L^2(0,T; L^2(\Omega))\,,
\]
so that we can study the auxiliary problem
\begin{align}
    \label{eq:conv-ch1}\partial_t v - \Delta \mu_v
    +\operatorname{div}(\beta_\lambda v) = 0
    \qquad&\text{in } (0,T)\times\Omega\,,\\
    \label{eq:conv-ch2}\mu_v = 
    \lambda\partial_t v
    -\lambda\Delta v
    +{\bf B}_\ep(w)
    +\gamma_\lambda(v)
    +\Pi(v)
    \qquad&\text{in } (0,T)\times\Omega\,,\\
    \label{eq:conv-ch3}v(0)=u_{0,\ep}^\lambda \qquad&\text{in } \Omega\,,
\end{align}
which
can be seen as a local convective viscous 
Cahn-Hilliard equation with 
an additional source term in the definition of the 
chemical potential. It is well-known
(see \cite{col-gil-spr-CHconv} for example)
that such problem
admits a unique weak solution $(v,\mu_v)$ with
\[
  v \in H^1(0,T; L^2(\Omega))\cap 
  L^\infty(0,T; H^1(\Omega)) \cap
  L^2(0,T; H^2(\Omega))\,, \qquad
  \mu_v \in L^2(0,T; H^1(\Omega))\,,
\]
satisfying \eqref{eq:conv-ch1}--\eqref{eq:conv-ch3} 
for example in the sense of distributions.
Hence, the map
\[
  \Gamma_\ep^\lambda:
  L^2(0,T; H^s(\Omega)) \to H^1(0,T; L^2(\Omega))
  \cap 
  L^\infty(0,T; H^1(\Omega))
  \cap L^2(0,T; H^2(\Omega))\, 
\]
associating to every $w\in L^2(0,T; H^s(\Omega))$ the solution $v$ to \eqref{eq:conv-ch1}--\eqref{eq:conv-ch3}
is well-defined. We proceed by showing that $\Gamma_\ep^\lambda$
has also some continuity properties.
For $i=1,2$ let $w_i\in L^2(0,T; H^s(\Omega))$, and set
$v_i:=\Gamma_\ep^\lambda(w_i)$. 
Then taking the difference of the corresponding equations \eqref{eq:conv-ch1} and \eqref{eq:conv-ch2}
for $i=1,2$, we obtain
\begin{align}
    &\label{eq:conv-ch1-diff}\partial_t (v_1-v_2) - \Delta (\mu_{v_1}-\mu_{v_2})
    +\operatorname{div}(\beta_\lambda (v_1-v_2)) = 0
    \qquad&\text{in } (0,T)\times\Omega\,,\\
    &\notag\mu_{v_1}-\mu_{v_2} = 
    \lambda\partial_t(v_1-v_2)
    -\lambda\Delta (v_1-v_2)
    +
    {\bf B}_\ep(w_1-w_2)\\
    &\label{eq:conv-ch2-diff}\qquad\quad\qquad+\gamma_\lambda(v_1)-\gamma_\lambda(v_2)
    +\Pi(v_1)-\Pi(v_2)
    \qquad&\text{in } (0,T)\times\Omega\,,\\
    &\label{eq:conv-ch3-diff}v_1(0)-v_2(0)=0 \qquad&\text{in } \Omega\,.
\end{align}
Noting that $(v_1-v_2)_{\Omega}=0$ by integrating \eqref{eq:conv-ch1-diff}, testing \eqref{eq:conv-ch1-diff} by $(-\Delta)^{-1}(v_1-v_2)$, equation \eqref{eq:conv-ch2-diff} by $v_1-v_2$, and taking the difference, estimate \eqref{estB} and assumption {\bf H4} yield
\begin{align*}
  &\|v_1-v_2\|_{C^0([0,t]; (H^1(\Omega))^*)
  \cap L^2(0,T; H^1(\Omega))}^2\\
  &\qquad \leq C_{\ep,\lambda}\Big\{
  \int_0^t\|w_1(s,\cdot)-w_2(s,\cdot)\|_{H^1(\Omega)}^2\,\xd s
  +\int_0^t\|\gamma_\lambda(v_1(s,\cdot))
  -\gamma_\lambda(v_2(s,\cdot))\|^2_{L^2(\Omega)}\,\xd s\\
  &\qquad+\int_0^t\|\Pi(v_1(s,\cdot))
  -\Pi(v_2(s,\cdot))\|^2_{L^2(\Omega)}\,\xd s + \int_0^t
  \|v_1(s,\cdot)-v_2(s,\cdot)\|_{L^2(\Omega)}^2\,\xd s\,\Big\},
\end{align*}
for every $t\in[0,T]$.

Testing \eqref{eq:conv-ch1-diff} by $v_1-v_2$, equation \eqref{eq:conv-ch2-diff} by $-\Delta (v_1-v_2)$, taking the difference, and using Lemma \ref{lem:dom_B}, a similar argument yields 
\begin{align*}
&\|v_1-v_2\|^2_{C^0([0,t];L^2(\Omega))}
+\|\Delta (v_1-v_2)\|^2_{L^2(0,t;L^2(\Omega))}\\
&\quad\leq C_{\ep,\lambda}\Big\{  \int_0^t\|w_1(s,\cdot)-w_2(s,\cdot)\|_{H^s(\Omega)}^2\,\xd s\\
&\qquad+\int_0^t\|\gamma_\lambda(v_1(s,\cdot))-\gamma_\lambda(v_2(s,\cdot))+\Pi(v_1(s,\cdot))-\Pi(v_2(s,\cdot))\|^2_{L^2(\Omega)}\,\xd s\\
&\qquad +\int_0^t\|v_1(s,\cdot)-v_2(s,\cdot)\|^2_{H^1(\Omega)}\,\xd s\Big\}.
\end{align*}
To handle the last term on the right-hand
side we use the following compactness result:
since $H^2(\Omega)\hookrightarrow H^1(\Omega)$ compactly,
for every $\eta>0$ there is $C_\eta>0$ such that 
\begin{align*}
  &\int_0^t\|v_1(s,\cdot)-v_2(s,\cdot)\|^2_{H^1(\Omega)}\,\xd s\\
  &\leq
  \eta\int_0^t
  \|\Delta(v_1-v_2)(s,\cdot)\|^2_{L^2(\Omega)}\,\xd s
  +C_\eta
  \int_0^t
  \|v_1(s,\cdot)-v_2(s,\cdot)\|^2_{L^2(\Omega)}
  \,\xd s\,.
\end{align*}

Hence, summing the two inequalities, using the Lipschitz-continuity of $\gamma_\lambda$
and $\Pi$,
choosing $\eta>0$ sufficiently small, and
applying the Gronwall's Lemma, we deduce that there exists a positive constant $C_{\ep,\lambda}$ such that 
\begin{equation}
\label{eq:cont}
  \|v_1-v_2\|_{C^0([0,T]; L^2(\Omega))
  \cap L^2(0,T; H^2(\Omega))}
  \leq C_{\ep,\lambda}
  \|w_1-w_2\|_{L^2(0,T; H^s(\Omega))}\,.
\end{equation}
In particular, $\Gamma_\ep^\lambda$ is continuous 
from $L^2(0,T; H^s(\Omega))$ to $L^2(0,T; H^s(\Omega))$.

Fix $T_0>0$. By repeating the argument leading to \eqref{eq:cont} we deduce the estimate
\begin{equation*}
  \|v_1-v_2\|_{C^0([0,T_0]; L^2(\Omega))
  \cap L^2(0,T_0; H^2(\Omega))}
  \leq C_{\ep,\lambda}
  \|w_1-w_2\|_{L^2(0,T_0; H^s(\Omega))}\,,
\end{equation*}
for every $w\in L^2(0,T_0;H^s(\Omega))$, and $v=\Gamma_{\ep}^{\lambda}(w)$.
Now, since $s\in(\frac32,2)$, if $\vartheta\in(0,1)$ is such that
\[
  s=(1-\vartheta)\cdot0 + \vartheta\cdot 2\,, \quad\text{i.e.}\quad
  \vartheta:=\frac{s}2\in\left(\frac34,1\right)\,,
\]
by interpolation we get that
\[
  \|v_1(t,\cdot)-v_2(t,\cdot)\|_{H^s(\Omega)}\leq\|v_1(t,\cdot)-v_2(t,\cdot)\|_{H^2(\Omega)}^{s/2}
  \|v_1(t,\cdot)-v_2(t,\cdot)\|_{L^2(\Omega)}^{1-s/2} \qquad\text{a.e.~in } (0,T)\,,
\]
which in turn yields that 
\begin{align*}
  \|v_1-v_2\|_{L^{4/s}(0,T_0; H^s(\Omega))}&\leq
  \|v_1-v_2\|_{L^2(0,T_0;H^2(\Omega))}^{s/2}
  \|v_1-v_2\|_{L^\infty(0,T_0;L^2(\Omega))}^{1-s/2}\\
  &\leq\frac{s}2\|v_1-v_2\|_{L^2(0,T_0;H^2(\Omega))}+
  \left(1-\frac{s}2\right)\|v_1-v_2\|_{L^\infty(0,T_0;L^2(\Omega))}\\
  &\leq C_s\|v_1-v_2\|_{C^0([0,T_0]; L^2(\Omega))
  \cap L^2(0,T_0; H^2(\Omega))}\,.
\end{align*}
Consequently, we have that 
\[
   \|v_1-v_2\|_{L^{4/s}(0,T_0;H^s(\Omega))} \leq
  C_{\ep,\lambda,s}\|w_1-w_2\|_{L^2(0,T_0; H^s(\Omega))}\,,
\]
where $\frac4{s}>2$ since $s<2$. Hence, we infer that
\[
\|v_1-v_2\|_{L^{2}(0,T_0;H^s(\Omega))} \leq T_0^{\frac{1}{2}-\frac{s}{4}}
\|v_1-v_2\|_{L^{4/s}(0,T_0;H^s(\Omega))}, 
\]
and we can choose $T_0$ sufficiently small such that $T_0^{\frac{1}{2}-\frac{s}{4}}C_{\ep,\lambda,s}<1$. Thus,
\[ 
\|v_1-v_2\|_{L^{2}(0,T_0;H^s(\Omega))} \leq
  T_0^{\frac{1}{2}-\frac{s}{4}}C_{\ep,\lambda,s}
  \|w_1-w_2\|_{L^2(0,T_0; H^s(\Omega))}.
\]
Banach fixed point theorem ensures the existence of a unique weak solution 
$(u_\varepsilon^{\lambda},\mu_\varepsilon^\lambda)$ 
to the approximated problem 
\eqref{1_app}-\eqref{3_app} in $(0,T_0)\times \Omega$, with
\[
u_\varepsilon^\lambda\in H^1(0,T_0; L^2(\Omega))
\cap 
  L^\infty(0,T_0; H^1(\Omega))
\cap L^2(0,T_0; H^2(\Omega)),\qquad \mu_\varepsilon^\lambda \in L^2(0,T_0;H^1(\Omega)).
\]
Note that the choice of $T_0$ 
is independent of the initial time.
Moreover, since 
$u_\ep^\lambda \in H^1(0,T_0; L^2(\Omega))
\cap L^\infty(0,T_0; H^1(\Omega))$, 
then $u_\ep^\lambda$ is weakly continuous 
with values in $T_0$: this allows us to 
obtain the pointwise regularity 
$u_\ep^\lambda(T_0)\in H^1(\Omega)$.
Such regularity is then 
enough to extend the 
solution to the next subinterval 
$[T_0,2T_0]$ (see \cite{col-gil-spr-CHconv}):
using a standard patching argument in time allows 
to extend the solution to the whole interval $[0,T]$.

\subsection{Uniform estimates}
\label{s:unif_est}
In this subsection we show that there exists $\ep_0>0$ independent of $\lambda$, and such that for $\ep<\ep_0$ the approximated solutions fulfill some uniform estimates 
independently of $\lambda$ and $\ep$. In what follows we will always assume that $\lambda\in [0,1]$.

{\em Step 1.} We start by fixing $t\in [0,T]$, testing \eqref{1_app} with  $\mu_\varepsilon^{\lambda}$, \eqref{2_app} with $\partial_t u_\varepsilon^{\lambda}$, taking the difference, and 
integrating the resulting equation on $(0,t)$.
We obtain
\begin{align*}
&\int_0^t\int_\Omega|\nabla\mu_\varepsilon^\lambda(s,x)|^2\,\xd x \,\xd s
+\lambda\int_0^t\int_\Omega
|\partial_t u_\ep^\lambda(s,x)|^2\,\xd x\,\xd s
+\frac\lambda2\int_\Omega|\nabla u_\ep^\lambda(t,x)|^2\,\xd x\\
&\qquad+ E_{\ep}(u^\lambda_\varepsilon(t,\cdot))
+ \int_\Omega (\hat\gamma_\lambda+\hat\Pi)(u^\lambda_\varepsilon(t,x))\,\xd x \\
&\leq \, \int_0^t\int_\Omega\beta_\lambda(t,x) u_\varepsilon^\lambda(t,x)\cdot
\nabla\mu^\lambda_\varepsilon(t,x)\,\xd x\,\xd t
+\frac\lambda2\int_\Omega|\nabla u_{0,\ep}^\lambda(x)|^2\,\xd x\\
&\qquad+E_{\ep}(u_{0,\ep}^\lambda)+ \int_\Omega (\hat\gamma_\lambda+\hat\Pi)(u_{0,\ep}^\lambda(x))\,\xd x.
\end{align*}
Using assumption \textbf{H3},
the uniform bound \eqref{ip_u0_ep_lam2} and
as well as Young's inequality, 
we get
\begin{align}
&\notag\int_0^t\int_\Omega|\nabla\mu_\varepsilon^\lambda(s,x)|^2\,\xd x\,\xd s + E_{\ep}(u_\varepsilon^\lambda(t,\cdot))
+\lambda\int_0^t\int_\Omega
|\partial_t u_\ep^\lambda(s,x)|^2\,\xd x\,\xd s
 +\frac\lambda2\int_\Omega|\nabla u_\ep^\lambda(t,x)|^2\,\xd x\\
&\label{eq:est1-high}\quad\leq  C_\ep
+\frac12 \int_0^t \int_\Omega |\nabla \mu^\lambda_\varepsilon(t,x)|^2\, \xd x\,\xd t 
+  \frac12\int_0^t \int_{\Omega}|\beta_{\lambda}(t,x)u_\varepsilon^\lambda(t,x)|^2\,\xd x\, \xd t 
\end{align}
for every $t\in [0,T]$.

We point out that, due to the periodic boundary conditions, and the fact that $\Omega$ is the $d$-dimensional torus, we formally have
\[
\int_{\Omega}\nabla {\bf B}_{\ep}(u_{\ep}^{\lambda}(s,x))\cdot\nabla u_{\ep}^{\lambda}(s,x)\,\xd x=
\frac12\int_{\Omega}\int_{\Omega}K_{\ep}(x,y)|\nabla u_{\ep}^{\lambda}(s,x)-\nabla u_{\ep}^{\lambda}(s,y)|^2\,\xd x\,\xd y
\]
for almost every $s\in [0,T]$.
Testing \eqref{1_app} with $u_{\ep}^{\lambda}$ and \eqref{2_app} with $-\Delta u_{\ep}^{\lambda}$, by considering the difference between the two resulting equation and by integrating in the time interval $(0,t)$, from \textbf{H3} we deduce the estimate
\begin{align*}
    &\frac12\int_{\Omega}|u_{\ep}^{\lambda}(t,x)|^2\,\xd x+\lambda \int_0^t\int_{\Omega}|\Delta u_{\ep}^{\lambda}(s,x)|^2\,\xd x\,\xd s
    +\int_0^t\int_{\Omega}\gamma_{\lambda}'(u_{\ep}^{\lambda}(s,x))|\nabla u_{\ep}^{\lambda}(s,x)|^2\,\xd x\,\xd s\\
    &\qquad
    +\frac\lambda2\int_\Omega
    |\nabla u_\ep^\lambda(t,x)|^2\,\xd x
    +\int_0^t\int_{\Omega}\int_{\Omega}K_{\ep}(x,y)|\nabla u_{\ep}^{\lambda}(s,x)-\nabla u_{\ep}^{\lambda}(s,y)|^2\,\xd x\,\xd y\,\xd s\\
    &\leq \frac12\int_{\Omega}|u_{0,\ep}^\lambda(x)|^2\,\xd x+\frac12\int_0^t\int_{\Omega}|\beta_{\lambda}(s,x)u_{\ep}^{\lambda}(s,x)|^2\,\xd x\,\xd s
    +\Big(C_{\Pi}+\frac12\Big)\int_0^t\int_{\Omega}|\nabla u_{\ep}^{\lambda}(s,x)|^2\,\xd x\,\xd s\\
    &\leq \frac12\int_{\Omega}|u_{0,\ep}^{\lambda}(x)|^2\,\xd x
    +\frac12\int_0^t\int_{\Omega}|\beta_{\lambda}(s,x)u_{\ep}^{\lambda}(s,x)|^2\,\xd x\,\xd s\\
    &\qquad+\frac14\int_0^t\int_{\Omega}\int_{\Omega}K_{\ep}(x,y)|\nabla u_{\ep}^{\lambda}(s,x)-\nabla u_{\ep}^{\lambda}(s,y)|^2\,\xd x\,\xd y\,\xd s+C\int_0^t\|u_{\ep}^{\lambda}(s,\cdot)\|^2_{L^2(\Omega)}\,\xd s,
\end{align*}
where the latter inequality holds for $\ep$ smaller than a suitable constant $\ep_0$ in view of Lemma~\ref{lemma:delta-est}. Noticing that the third term in the left-hand side of the above estimate is positive owing to the monotonicity of $\gamma_{\lambda}$, by \cite[Theorem 1.1]{ponce04} we infer the bound
\begin{align}
    \notag
    &\|u_\ep^\lambda(t,\cdot)\|_{L^2(\Omega)}^2+
    \|u_{\ep}^{\lambda}\|_{L^2(0,t;H^1(\Omega))}^2
    +\int_0^t\int_{\Omega}\int_{\Omega}K_{\ep}(x,y)|\nabla u_{\ep}^{\lambda}(s,x)-\nabla u_{\ep}^{\lambda}(s,y)|^2\,\xd x\,\xd y\,\xd s\\
    &\leq C\left(\int_{\Omega}|u_{0,\ep}^\lambda(x)|^2\,\xd x+\int_0^t\int_{\Omega}|\beta_{\lambda}(s,x)u_{\ep}^{\lambda}(s,x)|^2\,\xd x\,\xd s
    +\int_0^t\|u_{\ep}^{\lambda}(s,\cdot)\|^2_{L^2(\Omega)}\,\xd s\right) \label{eq:bd-H1-temp}.
\end{align}
By the H\"older inequality we deduce the estimate
\begin{equation}
    \label{eq:beta-lambda-reg}
    \int_0^t\int_{\Omega}|\beta_{\lambda}(s,x)u_{\ep}^{\lambda}(s,x)|^2\,\xd x\,\xd s\leq 
    \int_0^t\|\beta_{\lambda}(s,\cdot)\|^2_{L^\infty(\Omega)}\|u_{\ep}^{\lambda}(s,\cdot)\|^2_{L^2(\Omega))}\,\xd s.
\end{equation}
Thus, summing \eqref{eq:est1-high}, \eqref{eq:bd-H1-temp}, and \eqref{eq:beta-lambda-reg}, recalling \textbf{H4} we obtain
\begin{equation}\label{eq:est4-high}
\begin{split}
    &\|u_\ep^\lambda(t,\cdot)\|_{L^2(\Omega)}^2+
    \|u_{\ep}^\lambda\|^2_{L^2(0,t; H^1(\Omega))} + 
    \int_0^t\int_{\Omega}|\nabla \mu_{\ep}^{\lambda}(s,x)|^2\,\xd x\,\xd s\\
    &\qquad+E_{\ep}(u_{\ep}^{\lambda}(t,\cdot))+
    \int_0^t\int_{\Omega}\int_{\Omega}K_{\ep}(x,y)|\nabla u_{\ep}^{\lambda}(s,x)-\nabla u_{\ep}^{\lambda}(s,y)|^2\,\xd x\,\xd y\,\xd s\\
    &\leq C_\ep+C\|u_{\ep}^{\lambda}\|_{L^2(0,t;L^2(\Omega))}^2
    +C\|u_{\ep}^{\lambda}\|_{L^2(0,t;H^1(\Omega))}^2
    \|\beta\|^2_{L^2(0,T;L^\infty(\Omega;\mathbb{R}^d))}.
\end{split}
\end{equation}
Recalling assumption \textbf{H4} and
applying Gronwall's lemma, from the arbitrariness of $t\in [0,T]$ we deduce that
there exists a constant $C_\ep$ such that
\begin{align}
&\|\nabla\mu_\varepsilon^\lambda \|_{L^2(0,T;L^2(\Omega))}
\leq C_\ep \label{grad_mu}, \\
&\|u_\varepsilon^\lambda\|_{L^\infty(0,T;L^2(\Omega))\cap
L^2(0,T;H^1(\Omega))}
+\lambda^{1/2}\|u^\lambda_\ep\|_{L^\infty(0,T; H^1(\Omega))\cap
L^2(0,T; H^2(\Omega))\cap
H^1(0,T;L^2(\Omega))}
\leq C_\ep,
\label{u-linfty} \\
&\left \| E_\varepsilon(u^\lambda_\ep) \right\|_{L^\infty(0,T)}+
\left\|\int_{\Omega}\int_{\Omega}K_{\ep}(x,y)|\nabla u_{\ep}^{\lambda}(\cdot,x)-\nabla u_{\ep}^{\lambda}(\cdot,y)|^2\,\xd x\,\xd y\right\|_{L^1(0,T)}
\leq C_\ep.\label{conv-linfty}
\end{align}

Testing equation \eqref{1_app} with a function $\varphi\in L^2(0,T;H^1(\Omega))$, integrating in time, and using \eqref{grad_mu}--\eqref{conv-linfty} gives
\begin{equation}
    \label{estu_t2}
    \|\partial_t u^\lambda_\ep\|_{L^2(0,T;(H^1(\Omega))^*)}\leq C_\ep.
\end{equation} 

{\em Step 2.} In order to obtain an $L^2(0,T;H^1(\Omega))$-estimate on the chemical potential $\mu_\varepsilon^{\lambda}$, we need a bound on the $L^2(0,T)$-norm of the spatial mean of $\mu_\varepsilon^\lambda$. Thanks to the symmetry of the kernel $K$, the mean of the convolution terms vanishes, i.e.
\[
  ({\bf B}_\ep(u_\varepsilon^\lambda))_\Omega=0\,.
\]
Since also $(\Delta u_\ep^\lambda)_\Omega=0$,
owing to \eqref{u-linfty} and the Lipschitz continuity of $\Pi$,
we get
\begin{equation}
    \label{eq:est-mean}
(\mu_\varepsilon^\lambda)_\Omega = 
(\partial_t u_\ep^\lambda)_\Omega+
(\gamma_\lambda(u_\varepsilon^\lambda)+
\Pi(u_\ep^\lambda))_\Omega 
\leq C_\ep+\frac1{|\Omega|}\Vert \gamma_\lambda(u_\varepsilon^\lambda)\Vert_{L^1(\Omega)}.
\end{equation}
Hence $\{(\mu_\varepsilon^\lambda)_{\Omega}\}_{\ep}$
is uniformly 
bounded  in $L^2(0,T)$ if  $\{\gamma_\lambda(u_\varepsilon^\lambda)\}_{\ep}$ is uniformly bounded
in $L^2(0,T;L^1(\Omega))$. 
We test \eqref{1_app} by 
$(-\Delta)^{-1}(u_\varepsilon-(u_{0,\ep}^\lambda)_\Omega)$ 
and \eqref{2_app} by $u_\varepsilon-(u_{0,\ep}^\lambda)_\Omega$, 
obtaining
\begin{align*}
    &\quad \underbrace{\langle\partial_t u_\varepsilon^\lambda(t), (-\Delta)^{-1}(u_\varepsilon^\lambda(t,\cdot)
    -(u_{0,\ep}^\lambda)_\Omega)\rangle_{(H^1(\Omega))^*,H^1(\Omega)}}_{=:I_1}\\
    &+\underbrace{\lambda
    \langle\partial_t u_\varepsilon^\lambda(t), u_\varepsilon^\lambda(t,\cdot)
    -(u_{0,\ep}^\lambda)_\Omega)
    \rangle_{(H^1(\Omega))^*,H^1(\Omega)}
    +\lambda\int_{\Omega}|\nabla u_{\ep}^{\lambda}(t,x)|^2\,\xd x}_{=:I_2}\\
    &+ \underbrace{\int_\Omega {\bf B}_\ep(u_\varepsilon^\lambda)(t,x)
    (u_\varepsilon^\lambda(t,x)-(u_{0,\ep}^\lambda)_\Omega)\,\xd x}_{=:I_3} \\
    &\nonumber + \underbrace{\int_\Omega (\gamma_\lambda+\Pi)(u_\varepsilon^\lambda(t,x))
    (u_\varepsilon^\lambda(t,x)-(u_{0,\ep}^\lambda)_\Omega)\,\xd x}_{=:I_4}\\
    &- \underbrace{\int_\Omega \beta_\lambda(t,x) u_\varepsilon^\lambda(t,x)\cdot \nabla(-\Delta)^{-1}(u_\varepsilon^\lambda(t,x)
    -(u_{0,\ep}^\lambda)_\Omega)\,\xd x}_{=:I_5} = 0.
\end{align*}
We proceed by estimating each integral in the left-hand side of the above equation separately.

It is readily seen that $I_1+I_2$ 
is uniformly bounded in
$L^2(0,T)$
due to \eqref{u-linfty}, \eqref{estu_t2} and \eqref{ip_u0_ep_lam2}.

Regarding $I_3$, since $({\bf B}_\ep(u_\varepsilon^\lambda))_\Omega=0$
we have that 
\[
  I_3=2E_{\ep}(u_\ep^\lambda),
\]
which is clearly bounded in $L^2(0,T)$ by \eqref{conv-linfty}.

To estimate $I_4$ we observe that in view of \eqref{ip_u0_ep}
and \eqref{ip_u0_ep_lam}
there exist constants $M_1,M_2>0$ depending only on the position of 
$(u_{0,\ep})_\Omega$ in $\operatorname{Int}D(\gamma)$, such that 
\begin{equation*}
\gamma_\lambda(u_\varepsilon^\lambda)
(u_\varepsilon^\lambda-(u_{0,\ep}^\lambda)_\Omega)
\geq M_1|\gamma_\lambda(u_\varepsilon^\lambda)|-M_2,
\end{equation*}
cf.~for example \cite[p.~984]{col-gil-spr}
and the references within, while 
\[
  \int_\Omega \Pi(u_\ep(x))(u_\varepsilon(x)-
  (u_{0,\ep}^\lambda)_\Omega)\,\xd x
\]
is bounded in $L^\infty(0,T)$ thanks to \eqref{u-linfty}.

Eventually, $I_5$ can be estimated as follows  
\begin{equation*}
\| \beta_\lambda u^\lambda_\varepsilon\cdot \nabla(-\Delta)^{-1}
(u^\lambda_\varepsilon-(u_{0,\ep}^\lambda)_\Omega) \|_{L^2(0,T;L^1(\Omega))}^2\leq \int_0^T\Vert \beta_\lambda u_\varepsilon^\lambda\Vert^2_{L^2(\Omega)}\Vert (u_\varepsilon^\lambda-(u_{0,\ep}^\lambda)_\Omega)
\Vert^2_{(H^1(\Omega))^*}\,\xd t,
\end{equation*}
where the right-hand side is bounded due to \textbf{H4} and \eqref{u-linfty}.
\medskip

Combining this information, we conclude by difference that $\{\gamma_\lambda(u_\varepsilon^\lambda)\}$
is uniformly bounded in $L^2(0,T;L^1(\Omega))$. Thus, from \eqref{grad_mu} and \eqref{eq:est-mean} we infer that
\begin{align}
\Vert\mu_\varepsilon^\lambda\Vert_{L^2(0,T;H^1(\Omega))} & \leq C_\ep \label{est_mu}.
\end{align}

{\em Step 3.} We proceed by proving that $\{\gamma_\lambda(u_{\ep}^\lambda)\}$
is uniformly bounded in $L^2(0,T;L^2(\Omega))$.

We test \eqref{2_app} with 
$\gamma_\lambda(u_\varepsilon^\lambda)$. This gives 
\begin{align*}
\begin{aligned}
    &\int_0^T\int_\Omega 
    |\gamma_\lambda(u_\varepsilon^\lambda(t,x))|^2\,\xd x\, \xd t+\lambda\int_0^T\int_{\Omega}
    \gamma_{\lambda}'(u_{\ep}^{\lambda}(t,x))
    \left(
    |\partial_t u_\ep^\lambda(t,x)|^2+
    |\nabla u_{\ep}^{\lambda}(t,x)|^2\right)
    \,\xd x\,\xd t\\
    &\qquad + \int_0^T \int_\Omega {\bf B}_\ep(u_\varepsilon^\lambda)(t,x)
    \gamma_\lambda(u_\varepsilon^\lambda(t,x))\, \xd x \, \xd t \\
    &=\int_0^T\int_\Omega (\mu_\varepsilon^\lambda(t,x) -\Pi(u_\varepsilon^\lambda(t,x)))
    \gamma_\lambda(u_\varepsilon^{\lambda}(t,x))\, \xd x \, \xd t.
\end{aligned}
\end{align*}
We observe that the second term on the left-hand side is nonnegative owing to the monotonicity of $\gamma_{\lambda}$.  Analogously, the third term on the left-hand side can be rewritten as
\[
\int_0^T\int_\Omega\int_\Omega K_\varepsilon(x,y)(u_\varepsilon^\lambda(t,x)
-u_\varepsilon^\lambda(t,y))\Bigl(
\gamma_\lambda(u_\varepsilon^\lambda(t,x))-
\gamma_\lambda(u^\lambda_\varepsilon(t,y))\Bigr)\, \xd x \, \xd y \, \xd t,
\]
which is also nonnegative due to the monotonicity of 
$\gamma_\lambda$. 
Applying Young's inequality we deduce the bound
\begin{align*}
    &\quad \int_0^T\int_\Omega (\mu_\varepsilon^\lambda(t,x) -\Pi(u_\varepsilon^\lambda(t,x)))
    \gamma_\lambda(u_\varepsilon^\lambda(t,x))\, \xd x \, \xd t\\
    &\leq \int_0^T \int_\Omega \Big[\frac{|\mu_\varepsilon^\lambda(t,x) -\Pi(u_\varepsilon^\lambda(t,x))|^2}{2}
    + \frac{|\gamma_\lambda(u_\varepsilon^\lambda(t,x))|^2}{2}\Big]\, \xd x \, \xd t,
\end{align*}
which, together with \textbf{H3}, \eqref{u-linfty} and \eqref{est_mu}, implies the following estimate
\begin{equation}
\label{est_F'}
\Vert \gamma_\lambda(u^\lambda_\varepsilon)\Vert_{L^2(0,T;L^2(\Omega))}  \leq C_\ep. 
\end{equation}

\subsection{Passage to the limit as $\lambda\searrow0$}
\label{subs:conv}
We perform here the passage to the limit as $\lambda\searrow0$, 
with $0<\ep<\ep_0$ still fixed.
In view of the uniform bounds identified in Section~\ref{s:unif_est}
and the Aubin-Lions lemma,
up to the extraction of (not relabeled) subsequences we have the following convergences:
\begin{align}
 \label{conv_strong_lam}
 u_{\ep}^\lambda&\to u_\ep\quad & \text{strongly in }L^2(0,T;L^2(\Omega))\cap 
 C^0([0,T]; (H^1(\Omega))^*)\,,\\
 u_\varepsilon^\lambda
 &\rightharpoonup u_\ep & \text{ weakly* in } L^\infty(0,T;L^2(\Omega))\cap L^2(0,T;H^1(\Omega))\,, \label{conv_u_lam}\\
 \lambda u_\ep^\lambda &\to 0
 & \text{strongly in } L^\infty(0,T; H^1(\Omega))\cap L^2(0,T; H^2(\Omega))\,,\\
 \partial_t u_\varepsilon^\lambda &\rightharpoonup \partial_t u_\ep & \text{ weakly* in } L^2(0,T;(H^1(\Omega))^\ast)\,,\label{eq:wk-dt}\\
 \mu_\varepsilon^\lambda &\rightharpoonup \mu_\ep 
 & \text{ weakly in } L^2(0,T;H^1(\Omega))\,,\label{conv_mu_lam}\\
 \gamma_\lambda(u^\lambda_\varepsilon) &\rightharpoonup \xi_\ep & \text{ weakly in } L^2(0,T;L^2(\Omega)) \label{conv_gamma_lam},
\end{align}
for some
\begin{gather*}
u_\ep\in H^1(0,T;(H^1(\Omega))^\ast) \cap L^\infty(0,T;L^2(\Omega))\cap L^2(0,T;H^1(\Omega))\,,\\
\mu_\ep \in L^2(0,T;H^1(\Omega))\,, \qquad
\xi_\ep \in L^2(0,T;L^2(\Omega))\,.
\end{gather*}

The strong convergence \eqref{conv_strong_lam}, 
the weak convergence \eqref{conv_gamma_lam}
and the strong-weak closure of the maximal monotone 
graph $\gamma$ readily implies that 
$\xi_\ep\in\gamma(u_\ep)$ almost everywhere in $(0,T)\times \Omega$. The Lipschitz continuity of $\Pi$ yields also
\begin{equation}\label{conv_pi_lam}
  \Pi(u_\ep^\lambda) \to \Pi(u_\ep)
  \qquad\text{strongly in } L^2(0,T; L^2(\Omega))\,.
\end{equation}
Furthermore, for every $\varphi\in L^2(0,T; H^2(\Omega))$ 
by the triangle inequality we have that
\begin{align*}
  &\left|\int_0^T\int_\Omega\beta_\lambda(t,x)u^\lambda_\ep(t,x)\cdot\nabla
  \varphi(t,x)\,\xd x\,\xd t-
  \int_0^T\int_\Omega\beta(t,x)u_\ep(t,x)\cdot\nabla
  \varphi(t,x)\,\xd x\,\xd t\right|\\
  &\leq 
  \int_0^T\int_\Omega|\beta_\lambda(t,x)-\beta(t,x)||u^\lambda_\ep(t,x)|
  |\nabla\varphi(t,x)|\,\xd x\,\xd t\\
  &\quad+
  \int_0^T\int_\Omega\beta(t,x)(u^\lambda_\ep(t,x)-u_\ep(t,x))\cdot\nabla
  \varphi(t,x)\,\xd x\,\xd t\,.
\end{align*}
By the H\"older inequality, the fact that 
$\beta_\lambda\to \beta$ strongly in $L^2(0,T; L^3(\Omega))$
and the embedding $H^1(\Omega)\hookrightarrow L^6(\Omega)$,
for the first term on the right-hand side we have
\begin{align*}
  &\int_0^T\int_\Omega|\beta_\lambda(t,x)-\beta(t,x)||u^\lambda_\ep(t,x)|
  |\nabla\varphi(t,x)|\,\xd x\,\xd t\\
  &\qquad\leq
  \|u_\ep^\lambda\|_{L^\infty(0,T; L^2(\Omega))}
  \|\varphi\|_{L^2(0,T; L^6(\Omega)}
  \|\beta_\lambda-\beta\|_{L^2(0,T; L^3(\Omega))} \to 0\,.
\end{align*}
For the second term on the right-hand side note that 
$\beta\cdot\nabla\varphi\in L^1(0,T; L^2(\Omega))$
thanks to assumption \textbf{H4},
the fact that $\varphi\in L^2(0,T; H^2(\Omega))$ and the 
inclusion $H^1(\Omega)\hookrightarrow L^6(\Omega)$, so that 
from \eqref{conv_u_lam}
\[
\int_0^T\int_\Omega\beta(t,x)(u^\lambda_\ep(t,x)-u_\ep(t,x))\cdot\nabla
  \varphi(t,x)\,\xd x\,\xd t \to 0\,.
\]
Hence, we conclude that 
\[
  -\operatorname{div}
  \beta_\lambda u_\ep^\lambda \rightharpoonup 
  -\operatorname{div}
  \beta u_\ep\qquad
  \text{weakly* in } L^2(0,T; (H^2(\Omega))^*)\,.
\]
From \eqref{conv_u_lam} and the fact that ${\bf B}_\ep\in\mathscr{L}
(H^1(\Omega), (H^1(\Omega))^*)$, it is readily seen that 
\[
  {\bf B}_\ep(u_\ep^\lambda)\rightharpoonup {\bf B}_\ep(u_\ep)
  \qquad\text{weakly* in } L^2(0,T; (H^1(\Omega))^*)\,.
\]
By \eqref{conv_mu_lam}--\eqref{conv_gamma_lam}
and \eqref{conv_pi_lam},
by comparison it follows that the sequence 
$({\bf B}_{\ep}(u_{\ep}^{\lambda}))_{\lambda}$ is bounded
in $L^2(0,T;L^2(\Omega))$, hence we also conclude that
${\bf B}_\ep(u_\ep)\in L^2(0,T; L^2(\Omega))$
\[
 {\bf B}_\ep(u_\ep^\lambda)\rightharpoonup {\bf B}_{\ep}(u_\ep)
 \qquad\text{weakly in } L^2(0,T; L^2(\Omega))\,.
\]

Now, passing to the limit in \eqref{1_app}--\eqref{2_app}
as $\lambda\searrow0$, we obtain, in the sense of distributions,
\[
\partial_t u_\ep
-\Delta \mu_\ep=-\divergence(\beta u_\ep)
\]
and 
\[
\mu_\ep= {\bf B}_\ep(u_\ep)
+\xi_\ep + \Pi(u_\ep)\,.
\]
Finally, the strong convergence \eqref{conv_strong_lam}
implies also that $u_\ep(0)=u_{0,\ep}$, so that 
$(u_\ep,\mu_\ep,\xi_\ep)$ is a solution to the nonlocal Cahn-Hilliard equation \eqref{eq:NLCH-ad} according to 
Definition \ref{def:wsnl-ad}. 
This completes the proof of the first assertion of Theorem \ref{th1}.

\subsection{Continuous dependence}
Let $(\beta^1,u_{0,\ep}^1)$ and 
$(\beta^2,u_{0,\ep}^2)$ satisfy \textbf{H4} and
\eqref{ip_u0_ep}, with $(u_{0,\ep}^1)_\Omega=
(u_{0,\ep}^2)_\Omega$, and let 
$(u_\ep^1, \mu_\ep^1,\xi_\ep^1)$ and $(u_\ep^2,
\mu_\ep^2, \xi_\ep^2)$ be any corresponding 
solutions to the nonlocal equation \eqref{eq1:NL}--\eqref{eq3:NL}.
Then we have
\begin{align*}
    \partial_t(u_\ep^1-u_\ep^2)-\Delta(\mu_\ep^1-\mu_\ep^2)&=-\operatorname{div}(\beta^1u_\ep^1
    -\beta^2u_\ep^2)\,,\\
    \mu_\ep^1-\mu_\ep^2&={\bf B}_\ep(u_\ep^1-u_\ep^2)
    +\xi_\ep^1-\xi_\ep^2 + \Pi(u_\ep^1)-\Pi(u_\ep^2)\,.
\end{align*}
Noting that $(u_\ep^1-u_\ep^2)_\Omega=0$ by
the assumption on the initial data, we
test the first equation by $(-\Delta)^{-1}(u_\ep^1-u_\ep^2)$,
the second by $u_\ep^1-u_\ep^2$, and take the difference:
by performing classical computations we get
\begin{align*}
    &\frac12\|(u_\ep^1-u_\ep^2)(t)\|_{(H^1(\Omega))^*}^2
    +\int_0^t E_\ep(u_\ep^1-u_\ep^2)(s)\,\xd s
    +\int_0^t\int_\Omega
    (\xi_\ep^1-\xi_\ep^2)(s,x)(u_\ep^1-u_\ep^2)(s,x)
    \,\xd x\,\xd s\\
    &=\frac12\|(u_{0,\ep}^1
    -u_{0,\ep}^2)\|_{(H^1(\Omega))^*}^2
    -\int_0^t\int_\Omega
    (\Pi(u_\ep^1)-\Pi(u_\ep^2))(s,x)(u_\ep^1-u_\ep^2)(s,x)
    \,\xd x\,\xd s\\
    &\qquad+\int_0^t\int_\Omega
    \beta^1(s,x)(u_\ep^1-u_\ep^2)(s,x)\cdot\nabla
    (-\Delta)^{-1}(u_\ep^1-u_\ep^2)(s,x)
    \,\xd x\,\xd s\\
    &\qquad+\int_0^t\int_\Omega
    (\beta^1-\beta^2)(s,x)u_\ep^2(s,x)\cdot\nabla
    (-\Delta)^{-1}(u_\ep^1-u_\ep^2)(s,x)
    \,\xd x\,\xd s\,.
\end{align*}
By the Lipschitz-continuity of $\Pi$ we have
\[
  \int_0^t\int_\Omega
    (\Pi(u_\ep^1)-\Pi(u_\ep^2))(s,x)(u_\ep^1-u_\ep^2)(s,x)
    \,\xd x\,\xd s\leq
    C\|u_\ep^1-u_\ep^2\|^2_{L^2(0,t; L^2(\Omega))}\,,
\]
while the H\"older and Young inequalities yield
\begin{align*}
    &\int_0^t\int_\Omega
    \beta^1(s,x)(u_\ep^1-u_\ep^2)(s,x)\cdot\nabla
    (-\Delta)^{-1}(u_\ep^1-u_\ep^2)(s,x)
    \,\xd x\,\xd s\\
    &\leq\|u_\ep^1-u_\ep^2\|^2_{L^2(0,t; L^2(\Omega))}
    +\int_0^t
    \|\beta^1(s,x)\|_{L^\infty(\Omega)}^2
    \|(u_\ep^1-u_\ep^2)(s)\|_{(H^1(\Omega))^*}^2
    \,\xd x\,\xd s
\end{align*}
and
\begin{align*}
    &\int_0^t\int_\Omega
    (\beta^1-\beta^2)(s,x)u_\ep^2(s,x)\cdot\nabla
    (-\Delta)^{-1}(u_\ep^1-u_\ep^2)(s,x)
    \,\xd x\,\xd s\\
    &\leq\|\beta^1-\beta^2\|^2_{L^2(0,T; L^3(\Omega))}
    +\int_0^t\|u_\ep^2(s,\cdot)\|^2_{L^6(\Omega)}
    \|(u_\ep^1-u_\ep^2)(s)\|_{(H^1(\Omega))^*}^2\,\,\xd s.
\end{align*}
The continuous-dependence
property stated in Theorem~\ref{th1} follows from
Lemma~\ref{lemma:delta-est} and the Gronwall lemma.

\section{Proof of Theorem~\ref{th2}}
\label{proof2}
In this section we perform the limit as $\ep\searrow0$.

First of all, going back to the arguments performed in the previous 
section to obtain 
estimates \eqref{grad_mu}--\eqref{est_F'},
we observe that assumptions \eqref{ip_u0_ep_unif}--\eqref{ip_u0_ep_unif2} guarantee that the  sequence of constants $(C_\ep)_\ep$ is uniformly bounded for every $\ep\in(0,\ep_0)$.
Consequently, we deduce that there exists $C>0$
such that 
\begin{align}
    &\notag\|u_\ep\|_{H^1(0,T; (H^1(\Omega))^*)
    \cap L^2(0,T; H^1(\Omega))} &\leq C\,,\\
    &\label{eq:add19}\|E_\ep(u_\ep)\|_{L^\infty(0,T)} + 
    \left\|\int_{\Omega}\int_\Omega
    K_\ep(x,y)|\nabla u_\ep(x)-\nabla u_\ep(y)|^2\,\xd x\,\xd y
    \right\|_{L^1(0,T)} &\leq C\,,\\
    &\notag\|\mu_\ep\|_{L^2(0,T; H^1(\Omega))} &\leq C\,,\\
    &\notag\|\xi_\ep\|_{L^2(0,T; L^2(\Omega))} &\leq C\,.
\end{align}
Hence, by comparison
\[
  \|{\bf B}_\ep(u_\ep)\|_{L^2(0,T; L^2(\Omega))}\leq C\,.
\]
By Aubin-Lions compactness results we infer that, up to the extraction of (not relabeled) subsequences, 
\begin{align}
 \label{strong_L2}
 u_{\ep}&\to u\quad & \text{strongly in }L^2(0,T;L^2(\Omega))
 \cap C^0([0,T]; (H^1(\Omega))^*)\,,\\
 u_\varepsilon
 &\rightharpoonup u & \text{ weakly* in } L^\infty(0,T;L^2(\Omega))\,, \\
 \partial_t u_\varepsilon &\rightharpoonup \partial_t u & \text{ weakly* in } L^2(0,T;(H^1(\Omega))^\ast)\,,\\
 \bB_\varepsilon(u_\varepsilon)&\rightharpoonup
 \eta  &\text{ weakly in } L^2(0,T;L^2(\Omega))\,,\label{conv:B_ep}\\
 \mu_\varepsilon &\rightharpoonup \mu 
 & \text{ weakly in } L^2(0,T;H^1(\Omega))\,,\\
 \xi_\ep &\rightharpoonup \xi & \text{ weakly in } L^2(0,T;L^2(\Omega))
\end{align}
for some
\begin{gather*}
    u \in H^1(0,T; (H^1(\Omega))^*)\cap L^\infty(0,T; L^2(\Omega))\,,\\
    \mu \in L^2(0,T; H^1(\Omega))\,, \qquad
    \xi,\eta \in L^2(0,T; L^2(\Omega))\,.
\end{gather*}
We proceed by showing in addition that 
\begin{equation}\label{strong_H1}
    u^\lambda_\varepsilon \to u_\varepsilon \quad\text{strongly in }
    C^0([0,T]; L^2(\Omega))\cap L^2(0,T; H^1(\Omega))\,.
\end{equation}
Indeed,
Lemma \ref{lemma:delta-est} implies that
for every $\delta>0$, there exist $C_\delta >0$ and $\varepsilon_\delta>0$ such that 
\begin{align*}
  &
  \|u_\varepsilon-u\|_{L^2(0,T;H^1(\Omega))}^2\\
  &\leq \delta\int_0^T
  \int_\Omega\int_\Omega K_\ep(x,y)|\nabla (u_\varepsilon-u)(t,x)
  -\nabla (u_\varepsilon-u)(t,y)|^2 \,\xd x\, \xd y\,\xd t
  +C_\delta
  \|u_\varepsilon-u\|_{L^2(0,T;L^2(\Omega))}^2
\end{align*}
for every $0<\varepsilon<\varepsilon_\delta$.
Thanks to \eqref{eq:add19}, we infer that 
\[
  \|u_\varepsilon-u\|_{L^2(0,T;H^1(\Omega))}^2
  \leq C\delta + C_\delta\|u_\varepsilon-u\|_{L^2(0,T;L^2(\Omega))}^2
\]
for a constant $C>0$. 
Similarly, using the second inequality in Lemma~\ref{lemma:delta-est}
and \eqref{eq:add19}, 
the same argument ensures also that
\begin{align*}
  \|u_\varepsilon-u\|_{L^\infty(0,T;L^2(\Omega))}^2
  &\leq\delta \|E_\ep(u_\ep-u)\|_{L^{\infty}(0,T)} + C_\delta\|u_\ep-u\|^2_{C^0([0,T];
  (H^1(\Omega))^*)}\\
  &\leq C\delta + C_\delta\|u_\varepsilon-u\|_{C^0([0,T];
  (H^1(\Omega))^*)}^2\,.
\end{align*}
The strong convergence \eqref{strong_H1} follows then from the 
arbitrariness of $\delta$, and from \eqref{strong_L2}.\\

From the strong convergence of $(u_\ep)_\ep$ and
the strong-weak closure of maximal monotone graphs it is 
readily seen that $\xi \in \gamma(u)$ and that 
\[
  \Pi(u_\ep) \to \Pi(u) \qquad\text{strongly in } L^2(0,T; L^2(\Omega))\,.
\]
Let us now identify the limit $\eta$ as $-\Delta u$.
As $DE_\varepsilon=\bB_\varepsilon$, we have that
\begin{align}\label{eq:convex}
E_\varepsilon(z_1)+ \langle B_\varepsilon(z_1),z_2-z_1\rangle_{(H^1(\Omega))^*,H^1(\Omega)}&\leq E_\varepsilon (z_2),
\end{align}
for all $z_1,z_2\in H^1(\Omega)$.
Hence, for all $z\in L^2(0,T; H^1(\Omega))$ we deduce that
\begin{equation}
    \label{eq:add20}
    \int_0^TE_\varepsilon(u_\varepsilon(t,\cdot))\, \xd t
    +\int_0^T\int_\Omega 
    \bB_\varepsilon(u_\varepsilon(t,x))
    (z(t,x)-u_\varepsilon(t,x))\, \xd x \, \xd t\leq \int_0^T E_\varepsilon(z(t,\cdot))\, \xd t.
\end{equation}
The results in \cite{BBM}
and the dominated convergence theorem yield
\[
\int_0^T E_\varepsilon(z(t,\cdot))\, \xd t\to
\frac12\int_0^T\int_\Omega|\nabla z(x,t)|^2\xd x\,\xd t\,.
\]
Owing to the convergences~\eqref{strong_H1} and
\eqref{conv:B_ep}, we have that 
\[
\int_0^T\int_\Omega 
\bB_\varepsilon(u_\varepsilon(t,x))
(z(t,x)-u_\varepsilon(t,x))\, \xd x\, \xd t \to 
\int_0^T\int_\Omega\eta(t,x)(z(t,x)-u(t,x))\, \xd x\, \xd t.
\] 
Finally, following the exact same steps as in \cite{MRT18}, there holds
\[
\int_0^TE_\varepsilon(u_\varepsilon(t,\cdot))\,\xd t\rightarrow 
\frac12\int_0^T\int_\Omega|\nabla u(t,x)|^2\,\xd x\,\xd t.
\]

Hence, letting $\varepsilon\to0$
in \eqref{eq:add20}, we obtain the inequality 
\[
\frac12\int_0^T\int_\Omega|\nabla u(t,x)|^2\,\xd x\, \xd t
+\int_0^T\int_\Omega \eta(t,x) (z(t,x)-u(t,x))\, \xd x\, \xd t
\leq \frac12\int_0^T\int_\Omega|\nabla z(t,x)|^2\,\xd x\, \xd t
\]
for every $z\in L^2(0,T; H^1(\Omega))$, so that
$\eta=-\Delta u\in L^2(0,T;L^2(\Omega))$. By elliptic 
regularity we infer that $u\in L^2(0,T; H^2(\Omega))$.

Finally, 
H\"older's inequality, the Sobolev embedding 
$H^1(\Omega)\hookrightarrow L^6(\Omega)$, and the strong convergence \eqref{strong_H1} yield
\begin{align*}
  \|\beta u_\ep-\beta u\|_{L^2(0,T; L^2(\Omega))}
  &\leq\int_0^T\|\beta(t,\cdot)\|_{L^\infty(\Omega)}
  \|(u_\ep-u)(t,\cdot)\|_{L^2(\Omega)}\,\xd t\\
  &\leq
  \|\beta\|_{L^2(0,T; L^\infty(\Omega))}
  \|u_\ep-u\|_{L^\infty(0,T; L^2(\Omega))}\to0\,.
\end{align*}

Thus, letting $\ep\searrow0$
in Definition~\ref{def:wsnl-ad} (of solution for the nonlocal
Cahn-Hilliard) we obtain 
\[
  \partial_t u - \Delta\mu = -\operatorname{div}(\beta u)
\]
in the sense of distributions, as well as
\[
  \mu=-\Delta u + \xi + \Pi(u)\,.
\]
This implies that $u$ is a solution to the local Cahn-Hilliard equation \eqref{eq1:L}--\eqref{eq3:L}, 
and concludes the proof of Theorem \ref{th2}.

\section{Proof of Theorems~\ref{th3}--\ref{th4}}
\label{proof3}

We show first that under the additional assumption
\eqref{ip_u0_ep_reg}, the solution
$(u_\ep,\mu_\ep,\xi_\ep)$ to the nonlocal equation
is more regular. Note that here $\ep\in(0,\ep_0)$ is fixed.

The idea is to argue in a classical way, performing some
additional estimates on the approximate solutions
$(u_\ep^\lambda,\mu_\ep^\lambda)$ constructed in 
Section~\ref{proof1}. To this end, note that 
by \eqref{ip_u0_ep_reg}, the approximating 
sequence $(u_{0,\ep}^\lambda)_\lambda$ of initial data
satisfying \eqref{ip_u0_ep_lam}--\eqref{ip_u0_ep_lam2}
can be chosen with the additional property
\begin{equation}\label{ip_u0_ep_lam_reg}
  \sup_{\lambda\in(0,\lambda_0)}\left\{
  \|u_{0,\ep}^\lambda\|_{L^6(\Omega)}+\|-\lambda\Delta u_{0,\ep}^\lambda+{\bf B}_\ep(u_{0,\ep}^\lambda)
  +\gamma_\lambda(u_{0,\ep}^\lambda)+
  \Pi(u_{0,\ep}^\lambda)\|_{H^1(\Omega)}\right\}<+\infty\,.
\end{equation}

First of all
we need some preparatory work. Note that
the elliptic problem corresponding to 
\eqref{1_app}--\eqref{3_app} at time $0$, i.e.
\[
  \begin{cases}
  u_{0,\ep,\lambda}' - \Delta\mu_{0,\ep,\lambda} = 
  -\operatorname{div}(\beta(0)u_{0,\ep}^\lambda)\,,\\
  \mu_{0,\ep,\lambda}=
  -\lambda\Delta u_{0,\ep}^\lambda
  +{\bf B}_\ep(u_{0,\ep}^\lambda)
  +\gamma_\lambda(u_{0,\ep}^\lambda) + \Pi(u_{0,\ep}^\lambda)\,,
  \end{cases}
\]
admits a unique solution
$(u_{0,\ep,\lambda}', \mu_{0,\ep,\lambda})\in (H^1(\Omega))^*\times
H^1(\Omega)$. Testing the first equation by $\mu_{0,\ep,\lambda}$,
the second by $u_{0,\ep,\lambda}'$ and taking the difference yields
\begin{align*}
  \int_\Omega|\nabla \mu_{0,\ep,\lambda}(x)|^2\,\xd x&=
  -\langle u_{0,\ep,\lambda}',
  -\lambda\Delta u_{0,\ep}^\lambda
  +{\bf B}_\ep(u_{0,\ep}^\lambda)
  +\gamma_\lambda(u_{0,\ep}^\lambda) 
  + \Pi(u_{0,\ep}^\lambda)\rangle_{(H^1(\Omega))^*,H^1(\Omega)}\\
  &+ \int_{\Omega}\beta(0,x)u_{0,\ep}^\lambda(x)\cdot
  \nabla\mu_{0,\ep,\lambda}(x)\,\xd x.
\end{align*}
From the first equation it is readily seen that 
\[
  \|u_{0,\ep,\lambda}'\|_{(H^1(\Omega))^*}\leq
  \|\nabla\mu_{0,\ep,\lambda}\|_{L^2(\Omega)}
  +\|\beta(0)u_{0,\ep}^\lambda\|_{L^2(\Omega)}
\]
with
\[
  \|\beta(0)u_{0,\ep}^\lambda\|_{L^2(\Omega)}
  \leq\|\beta(0)\|_{L^3(\Omega)}\|u_{0,\ep}^\lambda\|_{L^6(\Omega)}
  \leq C\|\beta\|_{H^1(0,T;L^3(\Omega))}
  \|u_{0,\ep}^\lambda\|_{L^6(\Omega)}\,.
\]
Hence the Young inequality, \eqref{beta_reg}, \eqref{ip_u0_ep_lam}--\eqref{ip_u0_ep_lam2}, and \eqref{ip_u0_ep_lam_reg} imply that
\begin{equation}\label{est_mu0_lam}
  \|u_{0,\ep,\lambda}'\|_{(H^1(\Omega))^*}+
  \|\nabla \mu_{0,\ep,\lambda}\|_{L^2(\Omega)}
  \leq C_\ep\,.
\end{equation}

We are now ready to perform the additional estimate
on the approximated solutions. Again, we proceed formally
in order to avoid heavy notations and since 
everything can be proved rigorously through a further 
regularization on the problem.
The idea is to (formally) test the time derivative
of \eqref{1_app} by $(-\Delta)^{-1}(\partial_t u_{\ep}^\lambda)$, 
the time derivative of \eqref{2_app} by
$\partial_t u_\ep^\lambda$ and then to take the difference:
the resulting inequality is
\begin{align*}
    &\frac12\|\partial_t u_\ep^\lambda(t)\|_{(H^1(\Omega))^*}^2
    +\lambda\int_0^t\int_\Omega|\nabla\partial_t u_\ep^\lambda(s,x)|^2\,
    \xd x\,\xd s
    +\int_0^tE_\ep(\partial_t u_\ep^\lambda(s,\cdot))\,\xd s\\
    &\qquad+\int_0^t\int_\Omega\gamma_\lambda'(u_\ep^\lambda(s,x))
    |\partial_t u_\ep^\lambda(s,x)|^2\,\xd x\,\xd s
    +\int_0^t\int_\Omega\Pi'(u_\ep^\lambda(s,x))
    |\partial_t u_\ep^\lambda(s,x)|^2\,\xd x\,\xd s\\
    &=\frac12\|u_{0,\ep,\lambda}'\|_{(H^1(\Omega))^*}^2
    +\int_0^t\int_\Omega\partial_t u_\ep^\lambda(s,x)\beta(s,x)\cdot\nabla
    (-\Delta)^{-1}(\partial_t u_\ep^\lambda)(s,x)\,\xd x\,\xd s\\
    &\qquad
    +\int_0^t\int_\Omega u_\ep^\lambda(s,x)
    \partial_t\beta(s,x)\cdot\nabla
    (-\Delta)^{-1}(\partial_t u_\ep^\lambda)(s,x)\,\xd x\,\xd s.
\end{align*}

Now, note that by H\"older's inequality
and \eqref{beta_reg} we have
\begin{align*}
    &\int_0^t\int_\Omega\partial_t u_\ep^\lambda(s,x)\beta(s,x)\cdot\nabla
    (-\Delta)^{-1}(\partial_t u_\ep^\lambda)(s,x)\,\xd x\,\xd s\\
    &\qquad\leq\int_0^t\|\partial_t u_\ep^\lambda(s,\cdot)\|_{L^2(\Omega)}
    \|\beta(s,\cdot)\|_{L^\infty(\Omega)}
    \|\nabla(-\Delta)^{-1}(\partial_t u_\ep^\lambda)(s,\cdot)\|_{L^2(\Omega)}
    \,\xd s\\
    &\qquad\leq\frac12\|\partial_t u_\ep^\lambda\|^2_{L^2(0,t; L^2(\Omega))}
    +\frac12\int_0^t
    \|\beta(s,\cdot)\|_{L^\infty(\Omega)}^2
    \|\partial_t u_\ep^\lambda(s)\|_{(H^1(\Omega))^*}^2
    \,\xd s
\end{align*}
and
\begin{align*}
    &\int_0^t\int_\Omega 
    u_\ep^\lambda(s,x)\partial_t\beta(s,x)\cdot\nabla
    (-\Delta)^{-1}(\partial_t u_\ep^\lambda)(s,x)\,\xd x\,\xd s\\
    &\qquad\leq\int_0^t\|u_\ep^\lambda(s,\cdot)\|_{L^6(\Omega)}
    \|\partial_t\beta(s,\cdot)\|_{L^3(\Omega)}
    \|\nabla (-\Delta)^{-1}(\partial_t u_\ep^\lambda)\|_{L^2(\Omega)}
    \,\xd s\\
    &\qquad\leq\frac12\|u_\ep^\lambda\|^2_{L^2(0,t; H^1(\Omega))}
    +\frac12\int_0^t
    \|\partial_t \beta(s,\cdot)\|_{L^3(\Omega)}^2
    \|\partial_t u_\ep^\lambda(s)\|_{(H^1(\Omega))^*}^2\,\xd s.
\end{align*}
Thanks to Lemma~\ref{lemma:delta-est} there holds
\[
  \|\partial_t u_\ep^\lambda(s)\|_{(H^1(\Omega))^*}^2\leq\delta
  \int_0^tE_\ep(\partial_t u_\ep^\lambda)(s)\,\xd s
  +C_{\delta}\|\partial_t u_\ep^\lambda\|^2_{L^2(0,T; (H^1(\Omega))^*)}
\]
for $\delta$ sufficiently small. Hence, putting this information
together, using the Lipschitz-continuity of $\Pi$,
the monotonicity of $\gamma_\lambda$,
condition \eqref{est_mu0_lam} and 
the already proved estimates \eqref{u-linfty} and \eqref{estu_t2},
we are left with 
\begin{align*}
    &\|\partial_t u_\ep^\lambda(t)\|_{(H^1(\Omega))^*}^2
    +\int_0^tE_\ep(\partial_t u_\ep^\lambda(s,\cdot))\,\xd s\\
    &\leq C_\ep+ 
    \int_0^t\left(
    \|\beta(s,\cdot)\|_{L^\infty(\Omega)}^2
    +\|\partial_t\beta(s,\cdot)\|_{L^3(\Omega)}^2\right)
    \|\partial_t u_\ep^\lambda(s)\|_{(H^1(\Omega))^*}^2
    \,\xd s\,.
\end{align*}
Since $s\mapsto \|\beta(s,\cdot)\|_{L^\infty(\Omega)}^2$
and $s\mapsto \|\partial_t \beta(s,\cdot)\|_{L^3(\Omega)}^2$
belong to $L^1(0,T)$ due to \eqref{beta_reg} and \textbf{H4},
using the Gronwall lemma and 
recalling \cite[Theorem 1.1]{ponce} we infer that 
\begin{equation}
    \label{est_reg}
    \|\partial_tu_\ep^\lambda\|_{L^\infty(0,T; (H^1(\Omega))^*)\cap L^2(0,T; L^2(\Omega))} \leq C_\ep\,.
\end{equation}
Now, if \eqref{beta_reg2} holds,
we also have
\[
  \|\operatorname{div}(\beta u_\ep^\lambda)\|_{L^\infty(0,T; (H^1(\Omega))^*)}\leq
  \|\beta u_\ep^\lambda\|_{L^\infty(0,T; L^2(\Omega))}\leq
  \|\beta\|_{L^\infty(0,T; L^\infty(\Omega))}
  \|u_\ep^\lambda\|_{L^\infty(0,T; L^2(\Omega))}\,,
\]
yielding by \eqref{u-linfty} and by comparison in \eqref{1_app},
\begin{equation}
\label{est_reg2}
  \|\nabla\mu_\ep^\lambda\|_{L^\infty(0,T; L^2(\Omega))}\leq C_\ep\,.
\end{equation}
At this point, going back to the proof of Theorem~\ref{th1}, 
we repeat exactly the same arguments of {\em Step 2}~and~{\em Step 3}: using the additional estimates
\eqref{est_reg}--\eqref{est_reg2}, we deduce
\begin{equation}
    \label{est_reg3}
    \|\mu_\ep^\lambda\|_{L^\infty(0,T; H^1(\Omega))}+
    \|\gamma_\lambda(u_\lambda)\|_{L^\infty(0,T; L^2(\Omega))}\leq 
    C_\ep\,.
\end{equation}
Furthermore, if also \eqref{beta_reg3} holds we have
\begin{align*}
  \|\operatorname{div}(\beta u_\ep^\lambda)\|_{L^2(0,T; L^2(\Omega))}&\leq
  \|\operatorname{div}(\beta)u_\ep^\lambda\|_{L^2(0,T; L^2(\Omega))}+
  \|\beta\cdot\nabla u\|_{L^2(0,T; L^2(\Omega))}\\
  &\leq\|\operatorname{div}\beta\|_{L^\infty(0,T; L^3(\Omega))}\|u_\ep^\lambda\|_{L^2(0,T; L^6(\Omega))}\\
  &\qquad+
  \|\beta\|_{L^\infty(0,T; L^\infty(\Omega))}
  \|\nabla u_\ep^\lambda\|_{L^2(0,T; L^2(\Omega))}\,,
\end{align*}
so that from \eqref{estu_t2} and by comparison
in \eqref{1_app} we infer that 
\begin{equation}
    \label{est_reg4}
    \|\Delta\mu_\ep^\lambda\|_{L^2(0,T; L^2(\Omega))}
    \leq C_\ep\,.
\end{equation}
Hence, \eqref{est_reg}--\eqref{est_reg4} ensure
that the limit solution $(u_\ep,\mu_\ep,\xi_\ep)$
inherits the additional regularity stated in Theorem~\ref{th3}.
\\

The proof of Theorem~\ref{th4} follows now
as in Section~\ref{proof2}, noting that 
the assumption \eqref{ip_u0_ep_reg_unif}
implies that the family $(C_\ep)_\ep$
appearing in \eqref{est_reg}--\eqref{est_reg4}
is uniformly bounded in $\ep$.

\section*{Acknowledgements}

The authors are very grateful 
to the anonymous referee for the 
constructive suggestions and remarks.
E.D, H.R., and L.T. have been funded by the Austrian Science Fund (FWF) project F 65. 
The work of E.D. has been supported by the Austrian Science Fund (FWF) through projects I 4052-N32, and V 662-N32, as well as from BMBWF through the OeAD-WTZ project CZ04/2019.
L.T. acknowledges partial support from the Austrian Science Fund (FWF) project P27052.
L.S.~has been funded 
by Vienna Science and Technology Fund (WWTF) through Project MA14-009.

\bibliographystyle{abbrv}
\def\cprime{$'$}

\end{document}